\documentclass[11pt,a4paper,thmsb,ukenglish,fleqn]{article}
\usepackage{indentfirst,caption}
\usepackage{varioref,graphicx}
\usepackage{lipsum,amsmath,amssymb,amsthm,esdiff,eucal,mathtools,fixmath,mathrsfs,empheq, xcolor,bbm}
\usepackage{microtype}
\usepackage{geometry}
\usepackage{excludeonly}
\usepackage{epstopdf}
\usepackage[utf8]{inputenc}
\usepackage[sort]{natbib}
\usepackage{color}
\usepackage{a4wide}
\usepackage{appendix}

\usepackage{braket}

\newtheorem{lemma}{Lemma}[section]
\newtheorem{theorem}[lemma]{Theorem}
\newtheorem{proposition}[lemma]{Proposition}
\newtheorem{corollary}[lemma]{Corollary}

\newtheorem{definition}[lemma]{Definition}
\newtheorem{remark}[lemma]{Remark}
\newtheorem{example}[lemma]{Example}
\newtheorem{assumption}{Assumption}[section]

\allowdisplaybreaks
\newcommand{\filettata}{\mathbb}
\newcommand{\N}{\filettata{N}}

\newcommand{\Q}{\filettata{Q}}
\newcommand{\R}{\filettata{R}}
\newcommand{\Prob}{\filettata{P}}

\newcommand{\E}{\mathbb{E}}

\newcommand{\spl}{\begin{split}}
\newcommand{\lit}{\end{split}}

\newcommand{\cadlag}{c\`adl\`ag }

\newcommand{\tot}[1]{^{(#1)}}

\newcommand{\DD}{\Delta^n_i}

\newcommand{\JJ}{\Delta^n_j}
\newcommand{\si}{\sum_{i=1}^n}

\newcommand{\Var}{\text{Var}}

\newcommand{\qaz}{{(i-1)\Delta_n}}
\newcommand{\DDint}{\int_{(i-1) \Delta_n}^{i\Delta_n}}

\newcommand{\JJint}{\int_{(j-1) \Delta_n}^{j\Delta_n}}

\newcommand{\Et}{\E\left[}

\newcommand{\guno}{g^{(1)}(i\Delta_n-s)}
\newcommand{\gdue}{g^{(2)}(i\Delta_n-s)}
\newcommand{\deltag}[1]{\DD g^{(#1)}_s}

\newcommand{\BSS}{$\mathcal{BSS}$ }
\newcommand{\bss}{$\mathcal{BSS}$}

\newcommand{\leb}{\ensuremath{Leb}}

\allowdisplaybreaks

\renewcommand{\epsilon}{\varepsilon}
\renewcommand{\theta}{\vartheta}
\renewcommand{\phi}{\varphi}

{\left\lbrace\begin{array}{@{}l@{}}}%
{\end{array}\right.}

\DeclarePairedDelimiter{\abs}{\lvert}{\rvert}
\newcommand \norm[1]{\left\lVert #1 \right\rVert}

\DeclarePairedDelimiter\floor{\lfloor}{\rfloor}

\title{A weak law of large numbers for estimating the correlation in bivariate {B}rownian semistationary processes}
\author{Andrea Granelli\thanks{E-mail: \texttt{a.granelli12@imperial.ac.uk}} \and Almut E.~D.~Veraart\thanks{E-mail: \texttt{a.veraart@imperial.ac.uk}}}
\date{
\textit{Department of Mathematics, Imperial College London}\\
\textit{ 180 Queen's Gate, 
 London, SW7 2AZ, 
UK} \\ \ \\
\today}

\begin{document}
\maketitle
\begin{abstract}
This article presents various weak laws of large numbers for the so-called realised covariation of a bivariate stationary  stochastic process which is not a semimartingale. More precisely, we consider two cases: Bivariate moving average processes with stochastic correlation and   bivariate Brownian semistationary processes with stochastic correlation. In both cases, we can show that the (possibly scaled) realised covariation converges to the integrated (possibly volatility modulated) stochastic correlation process.
\end{abstract}


\section{Introduction}
The aim of this article  is to construct consistent estimators for the (possibly)  stochastic correlation between two stochastic processes \emph{outside} the semimartingale framework. 

In the semimartingale case, the corresponding results are well-known. 
For instance,  for a bivariate semimartingale ${\bf Y}=(Y^{(1)}, Y^{(2)})^{\top}$  the quadratic covariation denoted by $[Y^{(1)},Y^{(2)}]$ exists and can be approximated by 
the so-called realised covariation. More precisely, for $n\in\mathbb{N}$ we write $\Delta_n=n^{-1}$; then  we have for $t\geq 0$ that 
\begin{align*}
\sum_{i=1}^{\lfloor  n t \rfloor}\Delta_i^n Y^{(1)}\Delta_i^n Y^{(2)}\stackrel{u.c.p.}{\to} [Y^{(1)},Y^{(2)}]_t, \quad \text{ where }
\Delta_i^n Y^{(j)}=Y^{(j)}_{i\Delta_n}-Y^{(j)}_{(i-1)\Delta_n}, j\in \{1,2\},
\end{align*}
where the convergence is uniform on compacts in probability (u.c.p.) as $n\to \infty$. 
We know that as soon as we drop the assumption that ${\bf Y}$ is a semimartingale, its quadratic covariation does not necessarily exist anymore. 
Hence we would like to answer the question whether  a weak law of large numbers can be formulated for a possibly scaled version of the realised covariation 
$$
\sum_{i=1}^{\lfloor n t\rfloor}\Delta_i^n Y^{(1)}\Delta_i^n Y^{(2)}
$$
when ${\bf Y}$ is \emph{not} a semimartingale. We tackle this question in a semi-parametric setting: First of all, we define the class of bivariate moving average processes with stochastic correlation and later extend this class to bivariate Brownian semistationary processes with stochastic correlation. In these cases, we are able to show that, under appropriate assumptions,  the (possibly) scaled realised covariation can indeed be used  to estimate correlation in a non-semimartingale setting.

While our results are interesting from a purely theoretical point of view, we also believe that they are relevant for various applications. For instance, empirical work in  modelling of turbulence suggests that non-semimartingale models, and Brownian semistationary processes in particular, compare favourably to other alternatives, see e.g.~\cite{barndorff2009Brownian, MarSch15, BNBV2017Book}. Also, when it comes to modelling of financial asset prices, we remark that the classical theory has heavily relied on the semimartingale framework, but it has been shown that non-semimartingale models do not necessarily lead to the presence of arbitrage.  More precisely, in the presence of small proportional transaction costs, 
  \cite{GRS2008} showed that if a non-semimartingale has conditional full support, then it does not lead to so-called free lunches. This property is indeed satisfied by a Brownian semistationary processes, see \cite{pakkanen2011brownian}. Also, in the case of modelling energy prices the semimartingale assumption can be relaxed and (multivariate) Brownian semistationary processes have been used in this context, see e.g.~\cite{BNBV2013Spot,veraart2014modelling}. 

 Our work has been motivated by a variety of recent articles which consider weak laws of large numbers and central limit theorems for Gaussian processes, see \cite{GL1989, barndorff2009power, Corcuera2012New, barndorff2008bipower}, for Brownian semistationary processes, see \cite{barndorff2009Brownian, barndorff2011multipower}, and for L\'{e}vy-driven 
processes, see \cite{BLP2014, BasseOConnorHeinrichPodolskij2016}.
However, we remark that all these articles consider a univariate setting, whereas our article is to the best of our knowledge the first one to consider the multivariate case. This  enables us to construct for the first time a consistent estimator for stochastic correlation outside the semimartingale framework. 

The remainder of this article is structured as follows.  Section \ref{S2} introduces the notation and  defines the main objects of interest. In particular, it formulates the assumptions which ensure that we are outside the semimartingale framework. 
Section \ref{S3} contains the main contributions of the article by presenting three weak laws of large numbers for the (possibly scaled) realised covariation.
Since the proofs of our results are rather technical, we conclude in  
Section \ref{S4} and relegate the proofs to 
Section \ref{S5}. Finally, 
Section \ref{A} contains some useful  background material.

\section{The setting}\label{S2}
Throughout this article we denote by $\left(\Omega, \mathscr F,\mathscr F_t, \Prob\right)$  a filtered, complete probability space and by $\mathcal B(\R)$  the class of Borel subsets of $\R$. We will consider a finite time horizon $[0, T]$ for some $T>0$. 
Let us first recall the definition of a Brownian measure. 
\begin{definition}[Brownian measure]
An $\mathscr F_t$-adapted Brownian measure $W\colon \Omega\times\mathcal B(\R)\to \R$ is a Gaussian stochastic measure such that, if $A\in \mathcal B(\R)$ with $\E[(W(A))^2]<\infty$, then $W(A) \sim N(0,\leb(A))$,
where $\leb$ is the Lebesgue measure.
Moreover, if $A\subseteq [t,+\infty)$, then $W(A)$ is independent of $\mathscr F_t$.
\end{definition}
An introduction to constructing stochastic integrals against such measures can be found in \cite{walsh1986introduction}.

We will assume that $\left(\Omega, \mathscr F,\mathscr F_t, \Prob\right)$ supports two  independent $\mathscr F_t$-Brownian measures $W\tot1, \tilde W$ on $\R$ and will now define a bivariate moving average process with stochastic correlation.
\begin{definition}[Bivariate moving average process with stochastic correlation]
\label{otioti}
Consider two independent Brownian measures $W\tot1$ and $\tilde W$ adapted to $\mathscr F_t$ and 
two nonnegative deterministic functions $g\tot1, g\tot2 \in L^2((0,\infty))$ which are continuous on $\mathbb{R}\setminus\{0\}$. 
Let  $\rho$ be a \cadlag stochastic process, defined on the whole real line, with paths lying in $[-1,+1]$ a.s. and
independent of $\mathscr F_t$.
Define
\[
\begin{split}
&Y_t\tot1:= \int_{-\infty}^t g\tot 1(t-s)  \,dW\tot1_s,\\
& Y_t^{(2)}:=\int_{-\infty}^t g\tot 2(t-s) \rho_s \,dW\tot1_s+\int_{-\infty}^t g\tot2(t-s) \sqrt{1-\rho_s^2}\,d\tilde W_s.
\end{split} 
\]
Then the vector process $(\mathbf Y_t)_{t\geq 0}=(Y^{(1)}_t,Y^{(2)}_t)^{\top}_{t\geq 0}$
is called a bivariate moving average process with stochastic correlation.
\end{definition}

We call $\boldsymbol Y$ defined as above a bivariate moving average process with stochastic correlation, since we will formally write $dW\tot2_t:=\rho_sdW\tot1_t+\sqrt{1-\rho^2_t}d\tilde W_t$ and we can then write 
\[
Y^{(2)}_t=\int_{-\infty}^t  g^{(2)}(t-s)\,dW^{(2)}_s.
\]

Note that for $t>0$, $W\tot2 ([0,t]):=\int_0^t\rho_s\,dW\tot1_s+\int_0^t\sqrt{1-\rho_s^2}\,d\tilde W_s$ is a standard Brownian motion by L\'evy's characterisation theorem.
Hence $Y\tot 1$ and $Y\tot 2$ appear as two univariate moving average processes which feature stochastic dependence. 


In what follows we will mostly refer to the Brownian measures $W\tot1$ and $\tilde W$ as processes, by a slight abuse of notation.
An \emph{increment} $W_t-W_s$ will simply be the (Gaussian) random variable $W(\omega,(s,t])$.


\subsection{Integrated correlation and realised covariation}
Our object of interest is the integrated  stochastic correlation coefficient given by
$\int_0^t\rho_s ds$.
It is well known that in the case when ${\bf Y}$ is a semimartingale, then the quadratic covariation of ${\bf Y}$ is given by
\begin{align}\label{QC}
[Y^{(1)},Y^{(2)}]_t = g^{(1)}(0^+)g^{(2)}(0^+)\int_0^t\rho_s ds.
\end{align}
Our aim is to estimate (the right hand side of) equation \eqref{QC} consistently given high  frequency observations.
To this end consider the following setting. 
 Suppose  that we
 sample our processes discretely along successive partitions of $[0,T]$. A partition $\Pi_n$ of $[0,T]$ will be a collection of times $0=t_0<\dots<t_i<t_{i+1}<\dots<t_n=T$, where, for simplicity, we assume that the partition is equally spaced. The mesh of the partition will therefore be $\Delta_n=\frac 1n$. Hence when $n\to \infty$, $\Delta_n \to 0$ and we are in the setting of the so-called \emph{infill asymptotics}.

We denote the increments of  $Y^{(j)}$ by $\DD Y^{(j)}:=Y_{i\Delta_n}^{(j)}-Y_{(i-1)\Delta_n}^{(j)}$, for $j=1,2$, and find that
\begin{align}
\begin{split}
\DD Y^{(j)} &= \int_{-\infty}^{(i-1)\Delta_n} \left(g^{(j)}\left(i\Delta_n-s\right)-g^{(j)}\left(\qaz-s\right)\right)\,dW^{(j)}_s\\
&\qquad + \int_{(i-1)\Delta_n}^{i\Delta_n} g^{(j)}(i\Delta_n-s)\,dW^{(j)}_s.
\end{split}
\end{align}
The \emph{realised covariation} is defined as
\begin{align*}
\sum_{i=1}^{\lfloor nt \rfloor} \DD Y\tot1 \DD Y \tot2,
\end{align*}
for $n\geq 1, t\in [0,T]$. 
In the case when ${\bf Y}$ is a semimartingale, we know from  \cite{protter2005stochastic}(Theorem 23), that 
\begin{align*}
\sum_{i=1}^{\lfloor nt \rfloor} \DD Y\tot1 \DD Y \tot2 \overset{\text{u.c.p.}}{\rightarrow} [Y^{(1)},Y^{(2)}]_t, \quad \text{ as } n\to \infty,
\end{align*}
where the convergence is uniform on compacts in probability (u.c.p.), see Section \ref{Sectucp} for details on u.c.p.~convergence.

Since the asymptotic theory for realised covariation is well known in the semimartingale framework,  we are exclusively interested in the asymptotic behaviour of the realised covariation when ${\bf Y}$ is \emph{not} a semimartingale. In that case, we have no guarantee that, as $n$ tends to infinity, the realised covariation tends to a finite limit.
In the following we will prove
 convergence of the (possibly scaled) realised covariation and identify the limiting process.

\subsection{(Non-) semimartingale conditions}
The question of whether or not $Y^{(j)}$ for $j=1,2$ is a semimartingale hinges on the properties of the functions $g^{(j)}$.
To simplify the exposition, let us suppress the superscripts in the following and focus on the process
$Y_t=\int_{-\infty}^tg(t-s)dW_s$ in this subsection.  
In order to establish whether $Y$ belongs to the semimartingale class, we need to be precise about the filtration we are using.
There are three possible filtrations that can be considered: 1)
The filtration $\left(\mathscr F_t^{Y}\right)_{t\geq 0}$, the natural filtration of $Y$, i.e.~the smallest filtration to which $Y$ is adapted;
2) the filtration $\left(\mathscr F_t^{Y,\infty}\right)_{t\geq 0}$, such that $\mathscr F^{Y,\infty}_t:=\sigma\{Y_s  ,s\in(-\infty,t]\}$, i.e.~the filtration generated by the history of $Y$;
3) the filtration $\left(\mathscr F_t^{W,\infty}\right)_{t\geq0}$ which is the smallest filtration with respect to which $W$ is an adapted Brownian measure.

A discussion on the conditions to impose to ensure that $Y$ is a semimartingale in these different filtrations is contained in \cite{bassegaussian}, \cite{cheridito2004gaussian}. We will work with the condition that $Y$ is a semimartingale in the $\mathscr F_t^{W,\infty}$ filtration. This is the most restrictive of the three (see  \cite{bassegaussian}).
For this setting, there exists a classical theorem due to \cite{knight1992foundations} that states the following:
\begin{theorem}[Knight]
The process $(Y_t)_{t\geq 0}$ is an $\mathscr F_t^{W,\infty}$-semimartingale if and only if there exist $h \in L^2(\R)$ and $\alpha\in \R$ such that:
\[
g(t)=\alpha+\int_0^t h(s)\,ds.
\]
\end{theorem}

Next let us study some sufficient conditions for a more general class of stochastic processes to be a semimartingale: 
\cite{barndorff2009Brownian} extended the moving average process to allow for stochastic volatility and defined the so-called Brownian semistationary (\bss) process by
\begin{equation}
\label{definition}
X_t=\int_{-\infty}^t g(t-s)\,\sigma_s dW_s,
\end{equation}
where $W$ is an $\mathscr F_t$-adapted Brownian measure, $\sigma$ is \cadlag and $\mathscr F_t$-adapted, $g \colon \R \to \R$ is a deterministic function, continuous in $\R \setminus \{0\}$, with $g(t)=0$ if $ t\leq 0$ and $g\in L^2((0,\infty))$. Also it is assumed that $\int_{-\infty}^t g^2(t-s)\sigma_s^2\,ds<\infty$ a.s.~so that a.s.~we have $Y_t<\infty$ for all $t\geq 0$.
\cite{barndorff2009Brownian} gave the 
following sufficient conditions for \BSS process $X$ to be a semimartingale:
\begin{theorem}\label{SCforSM}
Under the assumptions that
\begin{enumerate}
\item[(i)] $g$ is absolutely continuous and $g'\in L^2((0,\infty)),$
\item[(ii)] $\lim_{x\to 0^+} g(x)=:g(0^+)<\infty,$
\item[(iii)] The process $g'(-\cdot)\sigma_\cdot$ is square integrable,
\end{enumerate}
then $X_t$ defined as in \eqref{definition} is an $\mathscr F^{W,\infty}_t$-semimartingale.
In this case $X_t$ admits the decomposition:
\[
X_t=g(0^+)W_t+\int_0^t\,dl \left[\int_{-\infty}^l g'(l-s)\sigma_s\,dW_s\right].
\]
\end{theorem}

We will now reintroduce the superscripts in our notation and formulate conditions which ensure that the bivariate moving average process ${\bf Y}$ with stochastic correlation is \emph{not} a semimartingale. 
To this end, we will relax the first two assumptions in Theorem \ref{SCforSM} since  both  assumptions are necessary for $Y^{(j)}$ to belong to the semimartingale class (see \cite{bassegaussian}) for $j=1,2$.

\begin{assumption}
\label{squareassumm}
For $j\in\{1,2\}$, we assume that $g\tot j\colon\R\to \R^+$ are nonnegative functions and continuous, except possibly at $x=0$. Also, $g\tot j(x)=0$ for $x<0$ and  $g\tot j \in L^2\left((0,+\infty\right))$.
We further ask that $g\tot j$ be differentiable everywhere with derivative $\left(g\tot j\right)'\in L^2((b^{(j)},\infty))$ for some $b^{(j)}>0$ and $\left((g\tot j)'\right)^2$ non-increasing in $[b^{(j)},\infty)$. 
\end{assumption}
To simplify the notation in some of the proofs we set $b=\max\{b^{(1)},b^{(2)}\}$, then  $\left(g\tot j\right)'\in L^2((b,\infty))$ and $\left((g\tot j)'\right)^2$ is non-increasing in $[b,\infty)$ for $j=1,2$.

Note that we are not assuming that $\left(g\tot j\right)'\in L^2((0,\infty))$ in order to exclude the semimartingale case. In particular, we must have that, for all $\epsilon >0$, $\sup_{x\in (0,\epsilon)} \left|\left(g\tot j\right)' (x)\right| = \infty.$

\begin{remark}
\label{marvellous} Under 
Assumption \ref{squareassumm} we can deduce from  the mean value theorem  that there exists an $\eta\in [s,s+\Delta_n]$ such that 
\[
g\tot j (s+\Delta_n)-g^{(j)}(s)= \Delta_n g\tot j\,'(\eta).
\]
 Then Assumption \ref{squareassumm} implies that, for $s\in [b,\infty)$, we have
\[
\abs{g\tot j (s+\Delta_n)-g\tot j(s)}\leq \Delta_n \abs{g\tot j\,'(s)}.
\]
Hence, we can  derive the bound:
\[
\int_b^{\infty} \left(g\tot j (s+\Delta_n)-g\tot j(s)\right)^2\,ds\leq \Delta_n^2\int_b^\infty \left(g\tot j\,'(s)\right)^2\,ds =K\Delta_n^2,
\]
for some constant $K>0$, since $\left(g\tot j\right)'\in L^2((b,\infty))$.
\end{remark}

We will impose an additional assumption for which 
 we recall here the definition of a slowly varying function.
\begin{definition}[Slowly and regularly varying function]
A measurable function $L\colon (0,\infty) \to (0,\infty)$ is called \emph{slowly varying at infinity} if, for all $\lambda >0$ we have that 
$\lim_{x\to\infty} \frac {L(\lambda x)}{L(x)} =1$.
A function $g\colon (0,\infty) \to (0,\infty)$ is called \emph{regularly varying at infinity} if, for $x$ large enough, it can be written as:
$g(x)=x^\delta L(x)$,
for a slowly varying function $L$. The parameter $\delta$ is called the \emph{index of regular variation}.
Finally, a  measurable function $L\colon (0,\infty) \to (0,\infty)$ is called \emph{slowly varying at zero} (resp. \emph{regularly varying at zero}) if $x\to L\left(\frac 1x\right)$ is slowly varying (resp. regularly varying) at infinity.
\end{definition}

The theory of regular variation allows us to construct functions which behave asymptotically like power functions. See also the discussion in \cite{bennedsen2015hybrid}, while the standard comprehensive reference on the subject is the  book by \cite{bingham1989regular}.

\begin{assumption}
\label{veryuseful}
We  require that, for $j\in\{1,2\}$, the following integral functions are regularly varying at zero:
\begin{align}
\int_0^{x} \left(g\tot j (s)\right)^2\,ds&=x^{2\delta\tot j +1}L_1\tot j(x), \label{aiutoo}
\\
 \int_0^{b\tot j} \left(g\tot j(s+x)-g\tot j (s)\right)^2\, ds&=x^{2\delta\tot j +1}L_2\tot j(x),
\end{align}
for  $\delta\tot j \in (-\frac12,0)\cup(0,\frac12)$, where $L_1\tot j(x)$ and $L_2\tot j(x)$ are positive,   slowly varying functions at zero which are continuous on $(0,\infty)$.
\end{assumption}
In this situation, the restriction $\delta\tot j\in (-\frac12,0)\cup(0,\frac12)$ ensures that the process leaves the semimartingale class.

\begin{example}
Suppose for simplicity $g(x)=x^\delta$ (where we suppress the superscripts again). Then we can  show that, if $\delta \in (-\frac 12,0)\cup (0,\frac 12)$, then 
\[
\int_0^{b} (g(x+\Delta_n)-g(x))^2\,dx = O(\Delta_n^{2\delta+1}).
\]
With the change of variable $x= y\Delta_n$ we get:
\[
\begin{split}
&\Delta_n \int_0^{bn}\left(g\left(\Delta_n(y+1)\right)- g\left(\Delta_n y\right)\right)^2\,dy=\Delta_n^{2\delta+1}\int_0^{bn} \left((y+1)^\delta-y^\delta\right)^2\,dy
 \\
&\leq \Delta_n^{2\delta+1}\int_0^\infty \left((y+1)^\delta-y^\delta\right)^2\,dy.
\end{split}
\]
It is then sufficient to show that $\int_0^\infty \left((y+1)^\delta-y^\delta\right)^2\,dy<\infty$.
First, suppose $\delta>0$. Then, in a neighbourhood of 0, $(y+1)^\delta-y^\delta \sim 1$, while by Taylor's theorem, $\left((y+1)^\delta-y^\delta\right)^2 \sim y^{2\delta-2}$ away from 0.
Hence, if $2\delta-2<-1\iff \delta<\frac 12$ the function is integrable at infinity.
If instead $\delta <0$, changing variables $z=\frac 1y$ and writing $-\delta=:\beta>0$, we obtain:
{\small\[
\int_0^\infty \left( \left(\frac{z+1}{z}\right)^\delta-\left(\frac 1z\right)^\delta\right)^2 \frac{1}{z^2}\,dz=\int_0^\infty \left(\frac{z^{\beta-1}\left(1-(z+1)^\beta\right)}{(z+1)^\beta}\right)^2\,dz.
\]}
For $z$ close to zero, $(z+1)^\beta\sim 1+z\beta$, so around 0 the integrand is asymptotically equivalent to $z^\beta\sim 0$.
At infinity, $\left(\frac{z^{\beta-1}\left(1-(z+1)^\beta\right)}{(z+1)^\beta}\right)^2\sim\left(\frac{z^{\beta-1}\left(-(z+1)^\beta\right)}{(z+1)^\beta}\right)^2=z^{2\beta-2}$, so we must impose $2\beta-2<-1\iff \beta <\frac 12 \iff \delta >-\frac 12.$
\end{example}

\begin{example} A popular example in applications is the so-called  Gamma kernel which is given by 
$g(x)=e^{-\lambda x} x^{\delta} 1_{\{x>0\}}$,  for $\lambda>0, \delta>-\frac{1}{2}$,
where superscripts are suppressed. 
 A review paper on the importance of the Gamma kernel in this context is given in \cite{barndorff2016assessing}. By way of example, the Gamma kernel is very relevant in turbulence  modelling, indeed $g(x)\sim x^{\delta}$, close to zero, and $\delta$ is called the \emph{scaling parameter}. A scaling parameter equal to $-\frac 16$ fits well with Kolmogorov's scaling law in turbulence  (see \cite{corcuera2013asymptotic}).

Note that with $\delta \in (-\frac12,0)\cup (0,\frac12],$ $g$ satisfies Assumptions \ref{squareassumm} and \ref{veryuseful} (a proof of the second statement is contained in \cite{barndorff2011multipower}). Indeed, for $\delta$ in the stated range, $g' \notin L^2(0,\infty)$, since it is unbounded and not square integrable near zero; in this case the moving average process (and also the Brownian semistationary process) is not a semimartingale.
 If $\delta > 0$, let  $b$ be the inflection point of $g$ and $\bar y$ its point of maximum. Then $b>\bar y$, and, if $y>b$, we have that $
((g')^2)'(y)=2g'(y)g''(y) \leq 0$,
since the Gamma kernel decreases after the maximum and the second derivative tends to $0+$ for $y\to \infty$.
If $\delta\leq0$, the condition that $(g')^2$ is non increasing in $[b,\infty)$ is satisfied  for any $b>0$.
Finally, $g'$ is obviously square integrable over $[b,\infty)$ thanks to its exponential decay.

\end{example}

The next  assumption we assume throughout concerns the stochastic correlation process $\rho$.
\begin{assumption}\label{As_rho}
The paths of $\rho$ are almost surely H\"older continuous with exponent  $\alpha$.\label{holderassum}
\end{assumption}

\begin{example}
If $\rho$ can be written as a diffusion:
$\rho_t=\int_0^t \kappa_s\,dW^*_s$ 
with $W^*$ a Brownian measure independent of $W\tot1,\tilde W$, and $\kappa$ satisfying some mild conditions,  then $\rho$ admits a version which is $\alpha-$H\"older continuous, for all $\alpha<\frac12$.
We can for example take the solution to the SDE:
\[
d\rho_t=\int_0^t \sqrt{(1-\rho_t)(1+\rho_t)}\,dW_t,
\]
which defines a stochastic process bound to stay between -1 and +1 (with the addition of an appropriate drift, such a process takes the name of \emph{Jacobi process}, see e.g.~\cite{veraart2012stochastic}).
\end{example}

\section{Results}\label{S3}
We will now state our main results where we distinguish  two different scenarios: The first scenario  considers the case when  
the  functions $g^{(j)}$ have a limit at zero, but the resulting process ${\bf Y}$ falls outside  the  semimartingale class. In the second scenario we will instead allow the  functions $g^{(j)}$ to blow up near zero, but we ask that they are decreasing.

\subsection{Weak law of large numbers for bivariate moving average processes with stochastic correlation}
\subsubsection{Scenario 1: The non-semimartingale case when   $g\tot1(0^+)g\tot2(0^+)<\infty$}
\label{asknunatsaesac}
This subsection presents the first of our convergence results for the realised covariation. 
Here we also impose an assumption on the limit behaviour of the functions $g^{(j)}$ at 0.
\begin{assumption}
\label{gizeroassum}
We suppose that $\delta\tot1+\delta\tot2\geq0$ and that the limit $g\tot 1 (0^+)g\tot2 (0^+):=\lim_{x\to 0^+}g\tot 1 (x)g\tot2 (x)$,
exists finite. If $\delta\tot1+\delta\tot2=0$, then we require that $\lim_{x\to 0^+}L_2\tot 1(x)L_2\tot2(x)=0$.
\end{assumption}
\begin{remark}
If we further assumed that:
\begin{equation}
\label{ultima?}
g\tot j(x) = x^{\delta\tot j} L_3 \tot j(x),
\end{equation} where $L_3$ is slowly varying, then Assumption \ref{gizeroassum} would tell us that $\lim_{x\to 0^+}g\tot 1 (x)g\tot2 (x)=0$. This is a common practical situation,  where both the kernels $g$ can be taken to be  $x^{\delta\tot j} e^{-\lambda\tot j x}$, for some positive, potentially different $\delta\tot1, \delta \tot2$.
Note, however, that  \eqref{aiutoo} is more general. Indeed, taking derivatives of $\eqref{aiutoo}$, we get:
\[
\left(g\tot{j}(x)\right)^2=(2\delta\tot j+1)x^{2\delta\tot j}L_1\tot j(x)+x^{2\delta\tot j+1}\left(L_1\tot j(x)\right)'.
\]
The derivative of a slowly varying function need not be slowly varying (see \cite{de1979derivatives}), hence we are not necessarily in the same context of \eqref{ultima?}.
%

\end{remark}

We can now formulate our first law of large numbers for the realised covariation of two non-semimartingales. 

\begin{theorem}[First law of large numbers]
\label{gizerotheo}
Suppose that the Assumptions \ref{squareassumm}, \ref{veryuseful}, \ref{As_rho} and \ref{gizeroassum}   are satisfied.
 Then, the following \emph{u.c.p.} convergence holds:
\begin{equation}
\label{convergencetoprove}
\left(\sum_{i=1}^{\floor{nt}} \DD Y\tot1 \DD Y^{(2)}\right)_{t\in[0,T]} \overset{\text{u.c.p.}}{\longrightarrow}\left(g^{(1)}(0^+)g^{(2)}(0^+)\int_0^t \rho_s\,ds\right)_{t\in[0,T]}, \quad \text{ as } n \to \infty.
\end{equation}
\end{theorem}

\subsubsection{Scenario 2: The non-semimartingale with decreasing functions $g^{(1)}, g^{(2)}$} \label{oihghn} 
In the case when  the limit $\lim_{x\to 0^+} g\tot1(x)g\tot2(x)$ does not exist finite we need to proceed differently in order to derive a weak law of large numbers for the realised covariation. A notable example of this scenario is the case when 
$g^{(j)}$ are Gamma kernels given by $g^{(j)}(x)=x^{\delta^{(j)}}e^{-\lambda^{(j)} x}$, for $ \delta^{(j)} <0, \lambda^{(j)}>0$ for $j=1,2$.

The method we are going to propose in the following is motivated by the method used in the univariate case by  \cite{barndorff2009Brownian}. To this end, let us  introduce the following notation.
We  write:
\[
\DD Y\tot1=\int_{-\infty}^{i\Delta_n} \phi^{(1)}_{\Delta_n}(i\Delta_n-s)\,dW^{(1)}_s,
\]
where 
\[ 
\phi_{\Delta_n}^{(j)}(s)=\begin{cases} g^{(j)} (s), \qquad &s\leq \Delta_n,\\ g^{(j)} (s)-g^{(j)} (s-\Delta_n), \qquad &s> \Delta_n,\end{cases} 
\]
and hence
\[ 
\phi_{\Delta_n}^{(j)}(i\Delta_n-s)=\begin{cases} g^{(j)} (i\Delta_n-s), \qquad &s\geq (i-1)\Delta_n,\\ g^{(j)} (i\Delta_n-s)-g^{(j)} (\qaz-s), \qquad &s< (i-1)\Delta_n.\end{cases}
\]

For reasons that will become apparent later, we need to impose monotonicity of our kernel functions. Since we want to be able to cover the case where $g$ goes to infinity close to zero, we introduce the following assumption:
\begin{assumption}
\label{monotonic}
We assume that $g\tot1, g \tot2$  are both  decreasing on $\R^+$.
\end{assumption}

We need to formulate another rather technical assumption which has first been formulated in \cite{GranelliVeraart2017b}. To this end let us introduce the necessary notation: We define the bivariate Gaussian process $({\bf G}_t)_{t\geq 0}$ whose components are given by $G^{(1)}_t=Y^{(1)}_t=\int_{-\infty}^t g^{(1)}(t-s)dW_s^{(1)}$ and $G^{(2)}_t=\int_{-\infty}^t g^{(2)}(t-s)dW_s^{(1)}$.
\begin{remark}
As a consequence of Assumption \ref{veryuseful},   if we denote for $i\in\{1,2\}$:
$$\bar R\tot i (t):=\E\left[\left(G\tot i_{s+t}-G\tot i_s\right)^2\right],$$ 
then it holds that:
$$\bar R\tot i (t)=t^{2\delta \tot i +1} L_0\tot i (t),$$
for $\delta\tot i\in\left(-\frac 12,\frac12\right)\setminus \{0\},$ where $ L\tot i_0(t)$ is  a slowly varying function at zero. Let us denote by $\tilde L\tot {i,j}_0(t):={\sqrt{ L\tot i_0(t) L\tot j_0(t)  }}$ another slowly varying function at zero.
\end{remark}

For $i,j\in \{1,2\}$, we write $\rho_{i,j}=\rho$ for $i\not =j$ and $\rho_{i,j}=1$ for $i=j$. Moreover, we  introduce the functions mapping $\R^+$ into $\R^+$, with 
$i,j\in\{1,2\}$:
\begin{equation}
\label{lplotk}
\bar R\tot {i,j} (t) :=\E\left[\left(G\tot j _{t}-G \tot i _0\right)^2\right]
=\norm{g\tot i}_{L^2}^2+\norm{g\tot j}_{L^2}^2-2\E\left[G \tot i_0G\tot j_{t}\right].
\end{equation}
We note that we can write
\begin{align*}
\bar R\tot {i,j} (t)
&=C_{i,j} + 2\rho_{i,j}\int_0^{\infty}(g^{(j)}(x)-g^{(j)}(x+t))g^{(i)}(x)dx,
\end{align*}
 where  $C_{i,j}:=\norm{g\tot i}_{L^2}^2+\norm{g\tot j}_{L^2}^2-2\rho_{i,j}\int_0^{\infty}g^{(i)}(x)g^{(j)}(x)dx$, where in particular $C_{i,i}=0$. 
We can now formulate  our next assumption.
\begin{assumption} \label{assss}
For all $t\in(0,T)$, there exist
slowly varying functions $L_0\tot {i,j}(t)$ and $L_2\tot{i,j}(t)$ which are  continuous on $(0,\infty)$ such that 
\begin{equation}\label{mfkjfkfid}
\bar R\tot {i,j} (t)= C_{i,j} + \rho_{i,j}
 t^{\delta\tot i+\delta\tot j+1}L_0\tot {i,j}(t), \quad \text{ for  } i,j\in\{1,2\},
 \end{equation}
and
\begin{equation*}
  \frac{1}{2}(\bar R\tot{i,j})''(t)=\rho_{i,j}t^{\delta\tot i+\delta\tot j-1} L_2\tot{i,j}(t), \quad \text{ for  } i,j\in\{1,2\},
\end{equation*}
where $\delta\tot 1, \delta\tot 2\in\left(-\frac 12,\frac12\right)\setminus \{0\}$.

Also, if we denote $\tilde L\tot {i,j}_0(t):={\sqrt{ L\tot{i,i}_0(t) L\tot{j,j}_0(t)  }}$,
we ask that  the functions $L_0\tot {i,j}(t)$ and $L_2\tot{i,j}(t)$ are such that, for all $\lambda >0$, { there exists a $H\tot{i,j}\in \R$ such that:}
\begin{equation}
\label{ewiue}
\lim_{t\to 0+} \frac{L_0\tot {i,j}(\lambda t)}{\tilde L_0\tot{i,j}(t)}= H\tot{i,j}<\infty,
\end{equation}
and that there exists $b\in (0,1)$, such that:
\begin{equation}\label{ass4}
\limsup_{x\to 0^+} \sup_{{y\in(x,x^b)}}\left|\frac{L\tot{i,j}_2(y)}{\tilde L_0^{(i,j)}(x)}\right|<\infty.
\end{equation}
\end{assumption}

\begin{remark}
We remark that Assumption \ref{assss} is important for the setting when $i\not =j$, in the case when $i=j$ the assumption is implied by the previous assumptions as discussed before.
\end{remark}

\begin{remark}
Assumption \ref{assss} can be seen as the extension of the classical assumptions on the variogram of a stationary process, see e.g.~\cite{corcuera2013asymptotic,barndorff2013limit, GL1989}.
\end{remark}
We need to introduce the following scaling factor for the realised covariation:
The scaling factor $c(x)$ is defined as
\begin{align*}
c(x)&:=\int_0^\infty\phi_{x}^{(1)}(s)\phi_{x}^{(2)}(s)\,ds\\
&=\int_0^{x} g\tot1(s)g\tot2(s)\,ds+\int_0^\infty\left(g\tot1(s+x)-g\tot1(s)\right)\left(g\tot2(s+x)-g\tot2(s)\right)\,ds.
\end{align*}
\begin{assumption}
\label{terrific}
The scaling factor
satisfies
\begin{equation}\label{pgkfjf}
c(\Delta_n)= \Delta_n^{\delta\tot1+\delta\tot2+1} L\tot{1,2}_4(\Delta_n),
\end{equation}
where $L_4\tot{1,2}$ is a continuous function on $(0,\infty)$ which is slowly varying  at zero and $\delta\tot 1, \delta\tot 2\in\left(-\frac 12,\frac12\right)\setminus \{0\}$. 
Moreover, we assume that there exists a constant $|H|<\infty$ such that 
\begin{align}\label{terrific2}
\lim_{x\to 0+}\frac{L_4^{(1,2)}(x)}{\tilde L_0^{(1,2)}(x)}=H.
\end{align}

\end{assumption}


Note that since the kernel functions are positive, decreasing and differentiable, $c(x)$ is a positive, increasing function of $x$.

We can now formulate the second weak law of large numbers for the scaled realised covariation of two non-semimartingales. 
\begin{theorem}
\label{weaktheonovol}
Suppose that Assumptions \ref{squareassumm}, \ref{veryuseful},  \ref{holderassum}, \ref{monotonic}, \ref{assss}  and  \ref{terrific} hold.
Then:
\[
\left(\Delta_n\frac{\sum_{i=1}^n \DD Y\tot1 \DD Y^{(2)}}{c(\Delta_n)}\right)_{t\in[0,T]}\overset{\text{u.c.p.}}{\longrightarrow} \left(\int_0^t \rho_l\,dl\right)_{t\in[0,T]}, \quad \text{ as } n \to \infty.
\]
\end{theorem}

We note that this result differs from the result obtained in Theorem \ref{gizerotheo} by the fact that we need to include the additional scaling factor $\frac{\Delta_n}{c(\Delta_n)}\sim \Delta_n^{-\delta\tot1-\delta\tot2}$. While this might seem innocent at first sight, it does impose restriction on empirical applications, since the scaling factor depends on the typically unknown functions $g\tot1 $ and $g\tot2$. 
Note further, that for $n$ large enough, the scaling factor $\frac{\Delta_n}{c(\Delta_n)}$ is increasing with $n$.

\subsection{Weak law of large number for bivariate Brownian semistationary processes with stochastic correlation}
So far we have considered bivariate moving average processes with stochastic correlation. In a next step, we wish to consider a more general class of stochastic processes which also allows for stochastic volatility. To this end, we give the  definition of the bivariate \BSS process with stochastic correlation that generalises Definition \ref{otioti}.

\begin{definition}[Bivariate Brownian semistationary process with stochastic correlation]
\label{2dimbss}
Consider two independent Brownian measures $W\tot1$ and $\tilde W$ adapted to $\mathscr F_t$ and 
two nonnegative deterministic functions $g\tot1, g\tot2 \in L^2((0,\infty))$ which are continuous on $\mathbb{R}\setminus\{0\}$. 
Let  $\rho$ be a \cadlag stochastic process, defined on the whole real line, with paths lying in $[-1,+1]$ a.s. and
independent of $\mathscr F_t$.
Let further $\sigma \tot1, \sigma\tot2$ be \cadlag, $\mathscr F_t$-adapted stochastic processes and assume that for $i\in\{1,2\}$, and for all $t\in [0,T]$:
\[
\int_{-\infty}^t g^{(i)2}(t-s) \sigma^{(i)2}_s\,ds<\infty.
\]
Define
\[
\begin{split}
&X_t\tot1:= \int_{-\infty}^t g\tot 1(t-s)  \sigma_s^{(1)}\,dW\tot1_s,\\
& X_t^{(2)}:=\int_{-\infty}^t g\tot 2(t-s) \sigma_s^{(2)}\rho_s \,dW\tot1_s+\int_{-\infty}^t g\tot2(t-s) \sigma_s^{(2)}\sqrt{1-\rho_s^2}\,d\tilde W_s.
\end{split} 
\]
Then the vector process $(\mathbf X_t)_{t\geq 0}=(X^{(1)}_t,X^{(2)}_t)^{\top}_{t\geq 0}$
is called a bivariate Brownian semistationary process with stochastic correlation.
\end{definition}

We note that we can write
\begin{align*}
{\bf X}_t &= \int_{-\infty}^t 
\left (\begin{array}{cc}g^{(1)}(t-s)&0\\
0&g^{(2)}(t-s) \end{array}\right)
\left (\begin{array}{cc}\sigma^{(1)}_s&0\\
0&\sigma^{(2)}_s \end{array}\right)
\left (\begin{array}{cc}1&0\\
\rho_s&\sqrt{1-\rho_s^2} \end{array}\right)
d \left (\begin{array}{c}W^{(1)}_s\\
\tilde W_s \end{array}\right),\\
&=\int_{-\infty}^t 
\left (\begin{array}{cc}g^{(1)}(t-s)&0\\
0&g^{(2)}(t-s) \end{array}\right)
\left (\begin{array}{cc}\sigma^{(1)}_s&0\\
0&\sigma^{(2)}_s \end{array}\right)
d \left (\begin{array}{c}W^{(1)}_s\\
 W^{(2)}_s \end{array}\right),
\end{align*}
where the integration is to be understood componentwise.
For the increment processes, we obtain for $j\in \{1, 2\}$:
\begin{multline}
\DD X\tot j= \int_{-\infty}^{(i-1)\Delta_n}\left( g\tot{j}(i\Delta_n-s)-g\tot{j}(i-1)\Delta_n-s)\right) \sigma\tot{j}_s\,dW\tot{j}_s\\+\int_{(i-1)\Delta_n}^{i\Delta_n} g\tot{j} (i\Delta_n-s)\sigma\tot{j}_s\,dW\tot{j}_s.
\end{multline}

Lastly, we will need the following technical assumption:
\begin{assumption}
\label{strangevolassumption}
The volatility processes are assumed to be independent of the Brownian measures, i.e.: $\sigma\left(\sigma\tot1,\sigma\tot2\right)$ is independent of $\sigma\left(W\tot1, \tilde W\right)$.
We assume that for $j\in\{1,2\}$, we have that the following property holds almost surely:
\[
\int_{1}^{\infty} \left(\frac{d}{ds} g\tot j(s)\right)^2 \sigma\tot j_{y-s}\,ds <\infty, \qquad \text{for all $y\in \R^+$}.
\]
\end{assumption}

We can write:
\[
\si \DD X^{(1)}\DD X^{(2)}=\si\int_0^\infty \phi\tot{1}_{\Delta_n}(s)\sigma\tot{1}_{i\Delta_n-s}\,dW\tot{1}_{i\Delta_n-s}\int_0^\infty \phi\tot{2}_{\Delta_n}(s)\sigma\tot{2}_{i\Delta_n-s}\,dW\tot{2}_{i\Delta_n-s}.
\]
We  now consider the sigma algebra $\mathscr H:=\mathscr F^{\rho,\sigma\tot{1},\sigma\tot 2}$ generated by the processes $\rho,\sigma\tot 1, \sigma \tot 2$. We obtain:
\begin{multline}
\E\left[\DD X\tot1 \DD X^{(2)} \Bigg| \mathscr H \right]=\int_0^\infty \phi_{\Delta_n}\tot 1\phi_{\Delta_n} \tot 2\sigma_{i\Delta_n-s} \tot1 \sigma_{i\Delta_n-s} \tot2 \rho_{i\Delta_n-s}\,ds\\
=\int_{-\infty}^{(i-1)\Delta_n} \DD g\tot1 \DD g\tot2 \sigma_{s} \tot1 \sigma_{s} \tot2 \rho_{s}\,ds+\int_{(i-1)\Delta_n}^{i\Delta_n} g\tot1(i\Delta_n-s)g\tot2(i\Delta_n-s)\sigma_{s} \tot1 \sigma_{s} \tot2 \rho_{s}\,ds.
\end{multline}

We aim to show the corresponding version of Theorem \ref{weaktheonovol}, in presence of volatility.

\begin{theorem}[Law of large numbers]\label{FullStory}
Suppose that the assumptions of Theorem \ref{weaktheonovol} hold.
Further assume that  the processes $\sigma\tot1,\sigma\tot2$
 have H\"older continuous sample paths and satisfy Assumption \ref{strangevolassumption}. Then
\label{conv2}
\[
\left(\Delta_n\frac{\sum_{i=1}^{\floor{nt}} \DD X\tot1 \DD X^{(2)}}{c(\Delta_n)}\right)_{t\in[0,T]}\overset{u.c.p.}{\longrightarrow}\left(\int_0^t \sigma\tot1_l\sigma\tot2_l\rho_l\,dl\right)_{t\in[0,T]}, \quad \text{ as } n \to \infty, 
\]
where $c(\Delta_n)=\int_0^{\Delta_n} g\tot1(s)g\tot2(s)\,ds+\int_0^\infty\left(g\tot1(s+\Delta_n)-g\tot1(s)\right)\left(g\tot2(s+\Delta_n)-g\tot2(s)\right)\,ds $.
\end{theorem}

\section{Conclusion}\label{S4}
In this article we have proven two versions of a law of large numbers for the realised covariation of the 
bivariate moving average process with stochastic correlation. We have also proven a weak law of large numbers for a bivariate Brownian semistationary process with stochastic correlation. 
These results state the consistency of our estimator for the integrated covariance, the (possibly scaled) realised covariation. The importance of such results can be seen in the more general context of building a fully multivariate theory for general \BSS process, whose importance in stochastic modelling in e.g.~finance and turbulence, see \cite{BNBV2013Spot,corcuera2013asymptotic}, is steadily rising. Performing inference on the dependence of two such processes is of fundamental importance in practical applications.
It has to be stressed, though, that our second law of large numbers requires us to scale the estimator by the typically unknown kernel functions $g\tot1 $ and $g\tot 2$, and this makes the estimation unfeasible in practice. Further research could be devoted to estimating the kernel functions first and use a plug-in type estimator as the scaling factor of the realised covariation we have considered in this article.


On the stochastic analysis side, our results  show that it is possible to obtain convergence of the ``realised covariation'' between two non-semimartingales,  subject to appropriate scaling. Our techniques are very general and one could perhaps try to adapt them to other classes of multivariate fractional process, for example extending the theory to include the integral processes appearing in \cite{corcuera2006power}.

In future research one could also tackle the problem of deriving a central limit theorem for bivariate Brownian semistationary processes and first results along these lines are available  in \cite{GranelliVeraart2017b}.


\section{Proofs}\label{S5}
\subsection{Proof of Theorem \ref{gizerotheo}}
\subsubsection{Notation and preliminary remarks}
This section contains the full proof of Theorem \ref{gizerotheo}, which will be divided into several separate lemmas.

Recall that we are working on  a  finite, fixed time horizon $[0,T]$, but according to  Theorem \ref{ucpsuff} and  Remark \ref{ucpremark} below, we do not lose generality if, for simplicity, we restrict ourselves to $T=1$. Hence we do so in the following.

\begin{remark}[From convergence in probability to $u.c.p.$ convergence]
\label{ucpremark}
We only need to show pointwise convergence in probability of the realised covariation in order to prove Theorem \ref{gizerotheo}.
Indeed, even though we are not directly in the framework of Theorem \ref{ucpsuff}, because the realised quadratic covariation process does not have increasing paths, since the product of the increments can  be negative, we can still deduce $u.c.p.$ convergence via the polarisation identity:
\[
\si\left( \DD Y \DD X\right)=\frac12 \si \left( \DD (X+Y) \right)^2-\si\left( \DD X \right)^2-\si\left( \DD Y \right)^2.
\]
If we can prove that the left hand side converges in probability for all $t$, it will follow that the realised quadratic variation of the sum $X+Y$ converges for all $t$, because we already know from the results in \cite{barndorff2009Brownian} that convergence of the squared increments of $X$ and $Y$ separately holds for all $t$.
But then each term on the right hand side converges $u.c.p.$ because they all have increasing paths. Since $u.c.p.$ convergence is induced by a metric, we can conclude that $u.c.p.$ convergence will hold for the realised quadratic covariation too.
\end{remark}

In the following,  the notation:
$\E_\qaz$
denotes the conditional expectation operator upon the sub-$\sigma$-algebra $\mathscr G_{(i-1)\Delta_n}$ generated by the whole process $\rho$ and by the increments of the processes $W^{(1)}$ and $\tilde W$ until $(i-1)\Delta_n$, i.e.:
\[
\mathscr G_{(i-1)\Delta_n}=\sigma\{\rho_s; s\in \R\}\vee\sigma\{W^{(1)}_u-W^{(1)}_t,\tilde W_u-\tilde W_t;\quad -\infty<t\leq u\leq \qaz\}.
\]
The main idea to prove \eqref{convergencetoprove} is to add and subtract the quantity:
\begin{equation}\label{subtract}
\si \E_\qaz \left[  \DD Y\tot1 \DD Y^{(2)} \right]=\E_\qaz \left[ \si \DD Y\tot1 \DD Y^{(2)} \right],\end{equation} and  split the difference $\si \DD Y\tot1 \DD Y^{(2)} -g^{(1)}(0^+)g^{(2)}(0^+)\int_0^1 \rho_s\,ds$ into:
 \begin{equation}A_n:=
\si \DD Y\tot1 \DD Y^{(2)}-\si \E_\qaz \left[  \DD Y\tot1 \DD Y^{(2)} \right], 
\label{split1}\end{equation}
\begin{equation}B_n:=
\si \E_\qaz \left[  \DD Y\tot1 \DD Y^{(2)} \right]-g^{(1)}(0^+)g^{(2)}(0^+)\int_0^1 \rho_s\,ds. 
\label{split2}\end{equation}
The proof that $A_n$  converges to zero will take most of the present section. Proposition \ref{ndvrhdvr} contains the much shorter proof that $B_n$ converges to zero as well.

Using the notation $\DD g^{(j)}_s:= g^{(j)}\left(i\Delta_n-s\right)-g^{(j)}\left(\qaz-s\right)$, the quantity in \eqref{subtract}  equals:
\begin{equation}
\label{conditionalsplit}
\begin{split}
&\E_\qaz \left[ \si \DD Y\tot1 \DD Y^{(2)} \right]\\
&=\si\E_\qaz\left[\DDint g^{(1)}\left(i\Delta_n-s\right) dW_s^{(1)} \DDint g^{(2)}\left(i\Delta_n-s\right)dW_s^{(2)}\right]\\
&+\si \int_{-\infty}^\qaz \DD g_s\tot1 \,dW_s^{(1)} \int_{-\infty}^\qaz \DD g_s\tot2 \,dW_s^{(2)}.
\end{split}
\end{equation}
We have to establish a preliminary lemma:
\begin{lemma}
\label{unkownlemma}
The following holds:
\begin{equation}
\label{ladshgaoe}
\begin{split}
&\E_\qaz \left[ \DDint g(i\Delta_n-s)\,dW_s\right]=0,\\
&\E_\qaz\left[\DDint g^{(1)}(i\Delta_n-s)\, dW^{(1)}_s \DDint g^{(2)}(i\Delta_n-s)\,dW^{(2)}_s\right]\\
&=\DDint g^{(1)}(i\Delta_n-s)g^{(2)}(i\Delta_n-s)\,\rho_s\,ds.
\end{split}
\end{equation}
\end{lemma}
\begin{proof}

The first result is immediate. 
For the second one, using the definition of $W^{(2)}$:
\begin{align}
&\E_\qaz\left[\DDint g^{(1)}(i\Delta_n-s)\, dW^{(1)}_s \DDint g^{(2)}(i\Delta_n-s)\,dW^{(2)}_s\right]\nonumber\\
=&\label{firstone}\E_\qaz\left[\DDint g^{(1)}(i\Delta_n-s)\, dW^{(1)}_s \DDint g^{(2)}(i\Delta_n-s)\,\rho_sdW^{(1)}_s\right]\\
+&\E_\qaz\left[\DDint g^{(1)}(i\Delta_n-s)\, dW^{(1)}_s \DDint \label{secondone} g^{(2)}(i\Delta_n-s)\,\sqrt{1-\rho^2_s}d\tilde W_s\right].
\end{align}
Line \eqref{secondone} immediately evaluates to zero by the tower property and independence of $W^{(1)},\tilde W$.

Observe that \[\sigma\left\{W^{(1)}_u-W^{(1)}_t,\tilde W_u-\tilde W_t;\quad -\infty<t\leq u\leq \qaz\right\}\] is independent of \[\sigma\left( \int_{(i-1)\Delta_n}^{i\Delta_n} g\tot1(i\Delta_n-s)dW\tot1_s\int_{(i-1)\Delta_n}^{i\Delta_n} g\tot2(i\Delta_n-s)\rho_sdW\tot1_s\right)\vee \sigma\left(\rho_s, s\in\R\right)=:\mathscr F^\rho.\]
Hence, conditioning on $\mathscr G_{(i-1)\Delta_n}$ is equivalent to conditioning on $\mathscr F^\rho$ (\cite{williams1991probability}).

By taking the derivative of the conditional characteristic function, and exchanging differentiation and expectation thanks to  uniform boundedness, we can write:
\[
\begin{split}
-\frac{\partial^2}{\partial\theta_1\partial\theta_2}\Bigg|_{\theta_1=0,\theta_2=0}\E\Biggl[&\exp\Biggl\{i\theta_1 \int_{(k-1)\Delta_n}^{k\Delta_n} g\tot1(k\Delta_n-s)\,dW\tot1_s\\&+i\theta_2 \int_{(k-1)\Delta_n}^{k\Delta_n} g\tot2(k\Delta_n-s)\rho_s\,dW\tot1_s\Biggr\}\Bigg| \mathscr F^\rho\Biggr]\\
=-\E\Biggl[\frac{\partial^2}{\partial\theta_1\partial\theta_2}\Bigg|_{\theta_1=0,\theta_2=0}&\exp\Biggl\{i\theta_1 \int_{(k-1)\Delta_n}^{k\Delta_n} g\tot1(k\Delta_n-s)\,dW\tot1_s\\&+i\theta_2 \int_{(k-1)\Delta_n}^{k\Delta_n} g\tot2(k\Delta_n-s)\rho_s\,dW\tot1_s\Biggr\}\Bigg| \mathscr F^\rho\Biggr]
\end{split}
\]\[
=\E \Biggl[\int_{(k-1)\Delta_n}^{k\Delta_n} g\tot1(k\Delta_n-s)\,dW\tot1_s\int_{(k-1)\Delta_n}^{k\Delta_n} g\tot2(k\Delta_n-s)\rho_s\,dW\tot1_s\Bigg| \mathscr F^\rho\Biggr],
\]
but on the other hand:
\[
\begin{split}
&-\frac{\partial^2}{\partial\theta_1\partial\theta_2}\Bigg|_{\theta_1=0,\theta_2=0}\E\Biggl[\exp\Bigg\{i\theta_1 \int_{(k-1)\Delta_n}^{k\Delta_n} g\tot1(k\Delta_n-s)\,dW\tot1_s\\&\qquad\qquad\qquad\qquad\qquad+i\theta_2 \int_{(k-1)\Delta_n}^{k\Delta_n} g\tot2(k\Delta_n-s)\rho_s\,dW\tot1_s\Bigg\}\Bigg| \mathscr F^\rho\Biggr]\\
&=-\frac{\partial^2}{\partial\theta_1\partial\theta_2}\Bigg|_{\theta_1=0,\theta_2=0}\E\Biggr[\exp\Biggl\{i\int_{(k-1)\Delta_n}^{k\Delta_n} \left[\theta_1 g\tot1(k\Delta_n-s)+\theta_2  \rho_s g\tot2(k\Delta_n-s)\right]\,dW\tot1_s \Biggr\}\Bigg| \mathscr F^\rho\Biggr]\\
&=-\frac{\partial^2}{\partial\theta_1\partial\theta_2}\Bigg|_{\theta_1=0,\theta_2=0}\exp\left(-\frac12 \int_{(k-1)\Delta_n}^{k\Delta_n}\left[ \theta_1 g\tot1(k\Delta_n-s)+\theta_2 \rho_s g\tot2(k\Delta_n-s)\right]^2\, ds\right)\\
&=\int_{(k-1)\Delta_n}^{k \Delta_n} g\tot1 (k\Delta_n-s) g\tot2 (k\Delta_n-s)\rho_s\,ds,
\end{split}
\]
proving the claim.
\end{proof}

We are now ready to begin to look at the proof of Theorem \ref{gizerotheo}. The next two sections contain the full details on the convergence of $A_n$ as defined in  \eqref{split1}.

\subsubsection{
Convergence of $A_n$}

First we show the convergence of $A_n$. To this end recall that 
\[ A_n=
\si \DD Y\tot1 \DD Y^{(2)}-\si \E_\qaz \left[  \DD Y\tot1 \DD Y^{(2)} \right]. 
\]

\label{l2section}

Expanding the conditional expectation using the computations in Lemma \ref{unkownlemma} and   simple algebra gives:
\begin{align}
&\si\Biggl( \DD Y\tot1 \DD Y^{(2)}-\E_\qaz \left[ \DD Y\tot1 \DD Y^{(2)} \right]\Biggr)\nonumber =A_n^{(1)}+A_n^{(2)},\quad \text{ where } \\
A_n^{(1)}&= \label{4.1}\si\int_{-\infty}^\qaz\deltag{1} dW^{(1)}_s\DDint g^{(2)}(i\Delta_n-s)dW^{(2)}_s \\
\label{4.2}&\quad+\si\int_{-\infty}^\qaz\deltag{2} dW^{(2)}_s\DDint g^{(1)}(i\Delta_n-s) dW^{(1)}_s\\
A_n^{(2)} &=\si\Biggl(\DDint g^{(1)}(i\Delta_n-s)\,dW^{(1)}_s \DDint g^{(2)}(i\Delta_n-s)\,dW^{(2)}_s \label{4.3}
\\ &\quad - \int_\qaz^{i\Delta_n}\guno\gdue \rho_s\,ds\Biggr).
\end{align}

\begin{proposition}
\label{isdughdv}
We have the following $L^2$-convergence:
$A_n^{(1)}\stackrel{L^2}{\to} 0$ as $n \to \infty$.
\end{proposition}
\begin{proof}

The two terms \eqref{4.1} and \eqref{4.2} in $A_n^{(1)}$ 
 are symmetric.  Hence we only present the proof that the first one converges to zero, and the proof for the second one will follow in an analogous way.
We  compute the $L^2$ norm of the first term in  $A_n^{(1)}$: 
\begin{equation}
\label{pougewwq}
\E\left[\left(\si \int_{-\infty}^\qaz\deltag{1} dW^{(1)}_s\DDint g^{(2)}(i\Delta_n-s)dW^{(2)}_s\right)^2\right].
\end{equation}
It is easy to see that the cross products originating by squaring the sum are equal to:
\begin{multline*}
2\sum_{j<i} \E\Biggl[\int_{-\infty}^\qaz\deltag{1} dW^{(1)}_s \int_{-\infty}^{(j-1)\Delta_n} \Delta^n_j g\tot1dW^{(1)}_s\\\times\DDint g^{(2)}(i\Delta_n-s)\,dW\tot2_s \JJint g^{(2)}(j\Delta_n-s)\,dW\tot2_s\Biggr],
\end{multline*}
which is easily seen to be zero, by the tower property, conditioning on $\mathscr G_{\qaz}$.
The other term originating from \eqref{pougewwq}, is the sum of the $n$ terms squared:
\begin{align}
&\si\E\left[\left(\int_{-\infty}^\qaz\deltag{1} dW^{(1)}_s \DDint g^{(2)}(i\Delta_n-s)\,dW\tot2_s\right)^2\right]\nonumber\\
=&\si \int_{-\infty}^\qaz\left(\deltag{1}\right)^2\,ds \DDint \left(g^{(2)}(i\Delta_n-s)\right)^2\,ds\label{ekugu},
\end{align}
thanks to an application of Corollary \ref{corollary5} in the Appendix.

Now \eqref{ekugu} becomes, after a change of variable:
\begin{equation}
\label{fulaewg}
n \int_0^\infty \left(g\tot1(s+\Delta_n)-g\tot1(s)\right)^2\,ds\int_0^{\Delta_n}\left(g\tot2(s)\right)^2\,ds.
\end{equation}

From Remark \ref{marvellous} it follows that it is only necessary to consider the bounded interval $(0,b\tot1)$ in the first integral of \eqref{fulaewg}, since on $[b\tot1,+\infty)$ the integral is $O(\Delta_n^2)$.

It then follows  from Assumption \ref{veryuseful} that the quantity in \eqref{fulaewg} can be written as follows, for some constant $C$:
\[
C n\Delta_n^{2\delta \tot1 + 2 \delta\tot2+2}L_1\tot 2(\Delta_n)L_2\tot 1(\Delta_n)=C \Delta_n^{2\delta \tot1 + 2 \delta\tot2+1}L_1\tot 2(\Delta_n)L_2\tot 1(\Delta_n) \to 0,
\]
since $\delta\tot1+\delta\tot2 \geq 0$, the product of slowly varying function is slowly varying and the Potter bounds apply, which we recall in Remark \ref{potter} below. 
\end{proof}

\begin{remark}[Potter's bound for slowly varying functions at zero]
\label{potter}
For a slowly varying function $L$ at infinity, we have Potter's bound: there exist constants $\bar u, C_1, C_2$ such that, for any $\delta>0$ and $u>\bar u$: $$C_1 u^{-\delta} \leq L(u) \leq C_2 u^\delta.$$
Call $M(u):= L(\frac 1u)$. M is slowly varying at zero. Potter's bound, for  small $u$, gives us:
\[
C_1 \left(\frac1u\right)^{-\delta} \leq L\left(\frac 1u\right) \leq C_2 \left(\frac 1u\right)^\delta,
\]
which we can rewrite as:
\[
C_1 u^{\delta} \leq M(u) \leq C_2 u^{-\delta} .
\]
In particular, for any $\alpha>0$, choose $\delta\in (0,\alpha)$. Then: \[0\leq t^\alpha M(t) \leq C_2 t^\alpha t^{-\delta}=C_2 t^{\alpha-\delta}\to 0
\]
as $t\to 0$. Henceforth, for all $\alpha>0$, $\lim_{t\to 0} t^\alpha M(t)=0$.
\end{remark}


For the first part of the proof of Theorem \ref{gizerotheo}, we are only left with the   term $A_n^{(2)}$, where we recall that  
\begin{multline*}
A_n^{(2)}=\si\DDint g^{(1)}(i\Delta_n-s)\,dW^{(1)}_s \DDint g^{(2)}(i\Delta_n-s)\,dW^{(2)}_s\\
- \int_\qaz^{i\Delta_n}\guno\gdue \rho_s\,ds.
\end{multline*}

Using that $dW_t\tot2=\rho_t\,dW\tot1+\sqrt{(1-\rho^2_t)}\,d\tilde W_t$, we will split it into two terms. I.e.~we will write $A_n^{(2)}=A_n^{(2,1)}+A_n^{(2,2)}$, where the first one is given by 
\[
A_n^{(2,1)}=\si\left(\DDint g^{(1)}(i\Delta_n-s)\,dW^{(1)}_s \DDint g^{(2)}(i\Delta_n-s)\sqrt{1-\rho_s^2}\,d\tilde W_s\right)
\]
and is dealt with in Lemma \ref{easylemma}.
The second one, whose convergence is more difficult to prove, is given by:
\begin{multline*}
A_n^{(2,2)}=\si\Biggl(\DDint g^{(1)}(i\Delta_n-s)\,dW^{(1)}_s \DDint g^{(2)}(i\Delta_n-s)\rho_s\,dW^{(1)}_s\\-\DDint(g^{(1)}(i\Delta_n-s))(g^{(2)}(i\Delta_n-s))\,\rho_s\,ds\Biggr).
\end{multline*}
A proof of its convergence can be found in Proposition \ref{finallyprop}.
We will start with the easier one:
\begin{lemma}\label{easylemma}
We have the following $L^2$-convergence:
$A_n^{(2,1)}\stackrel{L^2}{\to} 0$ as $n \to \infty$.
\end{lemma}
\begin{proof}
We get:
\begin{align}
&\E\left[\left(\si\DDint g^{(1)}(i\Delta_n-s)\,dW^{(1)}_s \DDint g^{(2)}(i\Delta_n-s)\sqrt{1-\rho_s^2}\,d\tilde W_s\right)^2\right] \\
= &\si\E\left[\left(\DDint g^{(1)}(i\Delta_n-s)\,dW^{(1)}_s \DDint g^{(2)}(i\Delta_n-s)\sqrt{1-\rho_s^2}\,d\tilde W_s\right)^2\right]+\label{5.3.0}\\
 & \sum_{i\ne j}\E\Bigg[\DDint g^{(1)}(i\Delta_n-s)\,dW^{(1)}_s \DDint g^{(2)}(i\Delta_n-s)\sqrt{1-\rho_s^2}\,d\tilde W_s  \times \label{5.3.1}\\
&\times \JJint g^{(1)}(j\Delta_n-s)\,d W^{(1)}_s \JJint g^{(2)}(j\Delta_n-s)\sqrt{1-\rho_s^2}\,d\tilde W_s  \Bigg] \nonumber.
\end{align}

Since $\mathscr F^{W\tot1}$ is independent of $\mathscr F^\rho \vee \mathscr F^{\tilde W} $, it is easy to deal with the term in \eqref{5.3.0}:
\begin{align}
&\si\E\left[\left(\DDint g^{(1)}(i\Delta_n-s)\,dW^{(1)}_s \DDint g^{(2)}(i\Delta_n-s)\sqrt{1-\rho_s^2}\,d\tilde W_s\right)^2\right]\nonumber\\
=&\si\E\left[\left(\DDint g^{(1)}(i\Delta_n-s)\,dW^{(1)}_s\right)^2\right] \E\left[\left(\DDint g^{(2)}(i\Delta_n-s)\sqrt{1-\rho_s^2}\,d\tilde W_s\right)^2\right]\\
=&\si \int_0^{\Delta_n} \left(g\tot1(s)\right)^2\,ds  \int_0^{\Delta_n}\left(g\tot2 (s)\right)^2\left(1-\rho^2_{i\Delta_n-s}\right)\,ds \\
\leq &  2 n\int_0^{\Delta_n} \left(g\tot1(s)\right)^2\,ds  \int_0^{\Delta_n}\left(g\tot2 (s)\right)^2  \,ds =  C\Delta_n^{2(\delta\tot1+\delta\tot2+1)} L_1\tot 1(\Delta_n)L\tot 2_1(\Delta_n) 
 n  \to 0, \label{5.3.2} 
\end{align}
for a constant $C>0$, where the asymptotics on line \eqref{5.3.2} follow from Assumption \ref{veryuseful} and Remark \ref{potter}.
Lastly, the summand in \eqref{5.3.1} becomes:
\[
\begin{split}
&\sum_{i\ne j}\E\Bigg[\DDint g^{(1)}(i\Delta_n-s)\,dW^{(1)}_s \DDint g^{(2)}(i\Delta_n-s)\sqrt{1-\rho_s^2}\,d\tilde W_s  \times \\
&\times \JJint g^{(1)}(j\Delta_n-s)\,d W^{(1)}_s \JJint g^{(2)}(j\Delta_n-s)\sqrt{1-\rho_s^2}\,d\tilde W_s  \Bigg]\\
=&\sum_{i\ne j}\E\Bigg[ \DDint g^{(2)}(i\Delta_n-s)\sqrt{1-\rho_s^2}\,d\tilde W_s  \JJint g^{(2)}(j\Delta_n-s)\sqrt{1-\rho_s^2}\,d\tilde W_s\Bigg] \times\\
& \E\Bigg[\DDint g^{(1)}(i\Delta_n-s)\,dW^{(1)}_s \times \JJint g^{(1)}(j\Delta_n-s)\,d W^{(1)}_s \Bigg]=0,
\end{split}
\]
since $((i-1)\Delta_n,i\Delta_n) \cap ((j-1)\Delta_n,j\Delta_n)=\emptyset$.
\end{proof}
Now we tackle the difficult term:
\begin{proposition}
\label{finallyprop}
We have the following convergence in probability:
$A_n^{(2,2)}\stackrel{\mathbb{P}}{\to} 0$ as $n \to \infty$.

\end{proposition}

We will use a \emph{blocking} technique, which consists in freezing the correlation process in each of the intervals $((i-1)\Delta_n, i\Delta_n)$. Exploiting  the H\"older continuity assumption on the paths of $\rho$, we will show that this ``frozen'' quantity is a good substitute of the one we have: We need to show that the difference converges to zero (at least) in probability.
The following lemma shows that there is essentially no harm in doing so for the Lebesgue integral part:
\begin{lemma} 
\label{lemmata1}
We have the following almost sure convergence:
\[\si \DDint g^{(1)}(i\Delta_n-s)g^{(2)}(i\Delta_n-s) \left(\rho_s-\rho_{(i-1)\Delta_n}\right)\,ds \stackrel{a.s.}{\to} 0, \quad \text{ as } n \to \infty.\]
\end{lemma}
\begin{proof}
We find that:
\[
\begin{split}
&\left| \si \DDint g^{(1)}(i\Delta_n-s)g^{(2)}(i\Delta_n-s) \left(\rho_s-\rho_{(i-1)\Delta_n} \right)\,ds\right| \\
&\leq \si \DDint g^{(1)}(i\Delta_n-s)g^{(2)}(i\Delta_n-s) \left|\rho_s-\rho_{(i-1)\Delta_n} \right|\,ds \\
&\leq \Delta_n^\alpha n \int_0^ {\Delta_n} g^{(1)}(s)g^{(2)}(s)\,ds \leq \Delta_n^\alpha n \sqrt{\int_0^ {\Delta_n}\left(g\tot1(s)\right)^2\,ds} \sqrt{\int_0^ {\Delta_n}\left(g\tot2(s)\right)^2\,ds} \\
&=C(\Delta_n^{\delta\tot1+\delta\tot2+\alpha}) \sqrt{L_1\tot1(\Delta_n)L_1\tot2(\Delta_n)} \to 0.
\end{split}
\]
\end{proof}
For the Brownian part, the situation is similar, but we cannot ask for a.s.~convergence. 
\begin{lemma}
\label{lemmata2}
We have the following $L^2$-convergence:
\[
\si \DDint g^{(1)}(i\Delta_n-s)\,dW^{(1)}_s \DDint g^{(2)}(i\Delta_n-s)\left(\rho_s-\rho_{(i-1)\Delta_n}\right)\,dW^{(1)}_s\stackrel{L^2}{\to} 0, \quad \text{ as } n \to \infty.
\]
\end{lemma}
\begin{proof}
Using the inequality $\E\left[ \left(\si X_i\right)^2\right] \leq n \si \E\left[X_i^2\right]$, we obtain:
\begin{align}
&\nonumber\E\left[\left(\si \DDint g^{(1)}(i\Delta_n-s)\,dW^{(1)}_s \DDint g^{(2)}(i\Delta_n-s)\left(\rho_s-\rho_{(i-1)\Delta_n}\right)\,dW^{(1)}_s\right)^2\right]\\
\nonumber\leq& n \si \E \left[ \E\left[ \left( \DDint g^{(1)}(i\Delta_n-s)\,dW^{(1)}_s
\right. \right.\right.
\\
& \left. \left. \left. \qquad 
 \DDint g^{(2)}(i\Delta_n-s)\left(\rho_s-\rho_{(i-1)\Delta_n}\right)\,dW^{(1)}_s\right)^2\Bigg|\mathscr F^\rho\right]\right] \nonumber\\
=&\nonumber n \si \E \Bigg[ \DDint \left(g\tot1 (i\Delta_n-s)\right)^2\,ds\DDint \left(g\tot2(i\Delta_n-s)\right)^2\left(\rho_s-\rho_{i\Delta_n-s}\right)^2\,ds \\
+&2\left(\DDint g\tot1 (i\Delta_n-s)g\tot2 (i\Delta_n-s) \left(\rho_s-\rho_{i\Delta_n-s}\right)\,ds\right)^2\Bigg], \label{5.6.1}
\end{align}
where the last equality \eqref{5.6.1}  follows from an application of Corollary \ref{corollary5} in the appendix.
Furthermore, by H\"older's inequality we have
\[
\begin{split}
&\left(\DDint g\tot1 (i\Delta_n-s)g\tot2 (i\Delta_n-s) \left|\rho_s-\rho_{i\Delta_n-s}\right|\,ds\right)^2\\
\leq& \left( \DDint \left(g\tot1 (i\Delta_n-s)\right)^2\,ds\DDint \left(g\tot2(i\Delta_n-s)\right)^2\left(\rho_s-\rho_{i\Delta_n-s}\right)^2\,ds\right)^2.
\end{split}
\]
Finally, using that $\abs{\rho_s-\rho_{(i-1)\Delta_n}}\leq C\Delta_n^{\alpha}$, for some positive constant $C$:
\[
\begin{split}
&3Cn\si \E \Bigg[ \DDint \left(g\tot1 (i\Delta_n-s)\right)^2\,ds\DDint \left(g\tot2(i\Delta_n-s)\right)^2\left(\rho_s-\rho_{i\Delta_n-s}\right)^2\,ds\Bigg]\\
\leq& n^2\Delta_n^{2\alpha} \int_0^{\Delta_n} \left(g\tot1(s)\right)^2\,ds\int_0^{\Delta_n} \left(g\tot2(s)\right)^2\,ds = \Delta_n^{2(\delta\tot1+\delta\tot2+\alpha)}L\tot 1_1(\Delta_n)L\tot 2_1(\Delta_n) \to 0,
\end{split}
\]
where the asymptotic bound follows again from Assumption \ref{veryuseful} and Remark \ref{potter}.
\end{proof}

\begin{proof}[Proof of Proposition \ref{finallyprop}]
We can now prove the statement of Proposition \ref{finallyprop}. In the light of Lemmas \ref{lemmata1} and \ref{lemmata2}, it will be sufficient to show that:
\begin{multline}
\label{takevariableoout}
\si  \Bigg(\DDint g\tot1 (i\Delta_n-s)\,dW\tot1_s \DDint g\tot2(i\Delta_n-s)\rho_{(i-1)\Delta_n}\,dW\tot1_s\\ - \DDint g\tot1 (i\Delta_n-s)g\tot2 (i\Delta_n-s)\rho_{(i-1)\Delta_n}\,ds\Biggr)
\end{multline}
converges to zero. We will show that it does so in $L^2$.
 Since we can take out the random variable $\rho_{(i-1)\Delta_n}$ from the stochastic integrals in \eqref{takevariableoout}, it remains to prove that:
\begin{multline}
\label{finalone}
\E\Bigg[\Bigg(\si \rho_{(i-1)\Delta_n} \Bigg(\DDint g\tot1 (i\Delta_n-s)\,dW\tot1_s \DDint g\tot2(i\Delta_n-s)\,dW\tot1_s \\- \DDint g\tot1 (i\Delta_n-s)g\tot2 (i\Delta_n-s)\,ds\Bigg)\Bigg)^2\Bigg]
\end{multline}
converges to zero.
For ease of notation, we will suppress the arguments of the functions $g\tot 1, g\tot 2$ in parts of the present proof.
By expanding \eqref{finalone} we obtain two terms: the sum of the squares and the sum of the cross products. We start with the latter:
\begin{multline}
\sum_{\substack{i,j=1\\i\ne j}}^n \E\Bigg[\rho_{(i-1)\Delta_n}\rho_{(j-1)\Delta_n} \times\Bigg(\DDint g\tot1 \,dW\tot1_s\DDint g\tot2 \,dW\tot1_s-\DDint g\tot1 g\tot2 \,ds\Bigg)\\\Bigg(\JJint g\tot1 \,dW\tot1_s\JJint g\tot2 \,dW\tot1_s-\JJint g\tot1 g\tot2\,ds\Bigg)\Bigg].
\end{multline}
The expectation splits thanks to independence between $\mathscr F^\rho$ and $\mathscr F^{W^{(1)}}$.
Now  suppose $i<j$; then by the independent increments property of the integrals:
\begin{align*}
&\E\Bigg[\Bigg(\DDint g\tot1 \,dW\tot1_s\DDint g\tot2 \,dW\tot1_s-\DDint g\tot1 g\tot2 \,ds\Bigg)\times\\ 
&\Bigg(\JJint g\tot1 \,dW\tot1_s\JJint g\tot2 \,dW\tot1_s-\JJint g\tot1 g\tot2\,ds\Bigg)\Bigg]\\
=&\E\Bigg[\Bigg(\DDint g\tot1 \,dW\tot1_s\DDint g\tot2 \,dW\tot1_s-\DDint g\tot1 g\tot2 \,ds\Bigg)\Bigg]\times\\
&\E\Bigg[ \Bigg(\JJint g\tot1 \,dW\tot1_s\JJint g\tot2 \,dW\tot1_s-\JJint g\tot1 g\tot2\,ds\Bigg)\Bigg]=0.
\end{align*}
Hence we are only left with considering the sum of squares:
\begin{align}\begin{split}
&\si \E\left[\rho_{(i-1)\Delta_n}^2\Bigg(\DDint g\tot1\,dW\tot1_s\DDint g\tot2\,dW\tot1_s-\DDint g\tot1g\tot2\,ds\Bigg)^2\right] 
\\
=&\si\E\left[\rho_{(i-1)\Delta_n}^2\right]\\
& \qquad \qquad \cdot\E\left[\Bigg(\DDint g\tot1\,dW\tot1_s\DDint g\tot2\,dW\tot1_s-\DDint \label{finalsum} g\tot1g\tot2\,ds\Bigg)^2\right].
\end{split}
\end{align}
Now:
\begin{align}
&\E\left[\left(\DDint g\tot1\,dW\tot1_s\DDint g\tot2\,dW\tot1_s-\DDint g\tot1g\tot2\,ds\right)^2\right]\nonumber\\\nonumber
&=\E\left[\left(\DDint g\tot1\,dW\tot1_s\DDint g\tot2\,dW\tot1_s\right)^2\right]-\left(\DDint g\tot1 g\tot2\,ds\right)^2\\
&=\DDint \left(g\tot1\right)^2\,ds \DDint \left(g\tot2\right)^2\,ds+\left(\DDint g\tot1 g\tot2\,ds\right)^2 \label{smartequality} \\
&\leq  2 \DDint \left(g\tot1\right)^2(s)\,ds \DDint \left(g\tot2\right)^2(s)\,ds
\\
&=C\Delta_n^{2(\delta\tot1+\delta\tot2+1)}L\tot1_1(\Delta_n)L\tot2_1(\Delta_n), \nonumber 
\end{align}
where \eqref{smartequality} is an application of Lemma \ref{important}. Thus,  \eqref{finalsum} finally goes to zero:
\begin{align*}
C\Delta_n^{2(\delta\tot1+\delta\tot2+1)}L\tot1_1(\Delta_n)L\tot2_1(\Delta_n) \si \E\left[\rho_{(i-1)\Delta_n}^2\right] &\leq C \Delta_n^{2(\delta\tot1+\delta\tot2)+1} L\tot1_1(\Delta_n)L\tot2_1(\Delta_n)\\
& \to 0,
\end{align*}
for some positive constant $C$.
The proof of Proposition \ref{finallyprop} is now complete. \end{proof}

\subsubsection{
Convergence of $B_n$} 

What remains to be shown is the convergence of the term $B_n$. Recall that 
\begin{equation}
\label{whatsleft}
B_n=
\si \E_\qaz \left[ \DD Y\tot1 \DD Y^{(2)} \right] - g^{(1)}(0^+)g^{(2)}(0^+)\int_0^1\rho_s\,ds, 
\end{equation}
where we recall the following expression, which we derived in \eqref{conditionalsplit} and \eqref{ladshgaoe}:
\[
\begin{split}
\si \E_\qaz \left[ \DD Y\tot1 \DD Y^{(2)} \right] = &\si  \int_{-\infty}^\qaz \deltag{1} dW^{(1)} \int_{-\infty}^\qaz \deltag{2}dW^{(2)} +\\& \si \DDint \guno \gdue \rho_s\,ds,
\end{split}
\]
with $\DD g^{(j)}_s:= g^{(j)}\left(i\Delta_n-s\right)-g^{(j)}\left(\qaz-s\right)$.
Hence we can write $B_n=B_n^{(1)}+B_n^{(2)}$, where 
\[
 B_n^{(1)}=\si  \int_{-\infty}^\qaz \DD g_s\tot1\,dW\tot1_s \int_{-\infty}^\qaz \DD g_s\tot2\,dW^{(2)}_s,
\]
and 
\[
B_n^{(2)}=\si \DDint \guno\gdue\rho_s\,ds -g^{(1)}(0^+)g^{(2)}(0^+) \int_0^1 \rho_s\,ds.
\]
\begin{lemma}
\label{thiswasaproblem}
We have the following $L^2$-convergence:
$B_n^{(1)}\stackrel{L^2}{\to} 0$ as $n \to \infty$.

\end{lemma}
\begin{proof}
Using the inequality $\E\left[ \left(\si X_i\right)^2\right] \leq n \si \E\left[X_i^2\right]$, it suffices to look at:
\begin{align}
\nonumber &n\si\E\left[ \left(  \int_{-\infty}^\qaz \DD g_s^{(1)}\,dW\tot1_s \int_{-\infty}^\qaz \DD g_s^{(2)} \,dW^{(2)}_s\right)^2 \right]\\\nonumber
&=n\si \int_{-\infty}^\qaz(\DD g_s\tot1)^2\,ds \int_{-\infty}^\qaz(\DD g_s\tot2)^2\,ds\\
& \quad+2n\si \left(\int_{-\infty}^\qaz (\DD g_s\tot 1)(\DD g_s\tot2)\rho_s\,ds\right)^2\\
\label{98.7}
&\leq 3n \si \int_{-\infty}^\qaz(\DD g_s\tot1)^2\,ds \int_{-\infty}^\qaz(\DD g_s\tot2)^2\,ds,
\end{align}
by H\"older's inequality.
Each of the two  integrals in \eqref{98.7} can be decomposed as:
\begin{multline}
\label{ooihds}
\int_{-\infty}^\qaz \left(\DD g_s\tot j \right)^2\,ds
\\=\int_0^{b^{(j)}} (g\tot j (s+\Delta_n)-g\tot j (s))^2\,ds+\int_{b^{(j)}}^{+\infty} (g\tot j (s+\Delta_n)-g\tot j (s))^2\,ds.
\end{multline}

The integral on the unbounded domain is $O(\Delta_n^2)$, as in Remark \ref{marvellous}.
Hence it suffices to consider the product of the integrals in the bounded domains, i.e.:
\begin{equation}
\label{porehgs}
n\si\int_0^{b\tot1} (g\tot 1 (s+\Delta_n)-g\tot 1 (s))^2\,ds\int_0^{b\tot2} (g\tot 2 (s+\Delta_n)-g\tot 2 (s))^2\,ds.
\end{equation}
Thanks to Remark \ref{potter} and Assumption \ref{veryuseful}, the quantity in \eqref{porehgs} becomes, for some $C>0$:
\[
 C n^2 \Delta_n^{2+2(\delta\tot1+\delta \tot2)} L\tot1_2(\Delta_n)L\tot2_2(\Delta_n)\to0.
\]

\end{proof}
\begin{remark}
Note that the last bound is sharp, in the sense that we obtained $O(\Delta_n^{2(\delta\tot1+\delta\tot2)})$ instead of $O(\Delta_n^{2(\delta\tot1+\delta\tot2+\alpha)})$. This is the only term in the proof where we cannot obtain more regularity by the H\"older continuity of the paths of $\rho$. Hence,  Assumption \ref{gizeroassum} which we are imposing in this section is crucial here.
\end{remark}
We conclude with the last part of the convergence theorem. 
\begin{proposition}
\label{ndvrhdvr}
We have the following almost sure convergence result: $B_n^{(2)}\stackrel{a.s.}{\to} 0$ as $n\to \infty$.
\end{proposition}
\begin{proof}
Using Lemma \ref{lemmata1} we can simply show that 
\[
\si \rho_{(i-1)\Delta_n}\DDint \guno\gdue\,ds
\]
converges almost surely to the limit $g^{(1)}(0^+)g^{(2)}(0^+)\int_0^1\rho_s\,ds$. To this end, note that 
\begin{align}
&\label{1.1}\si \rho_{(i-1)\Delta_n}\DDint \guno\gdue\,ds 
\\\label{1.2} &=\si \rho_{(i-1)\Delta_n} \int_0^{\Delta_n} g^{(1)}(y) g^{(2)}(y)\,dy
\\ \label{1.3}&= \left(\int_0^{\Delta_n} g^{(1)}(y) g^{(2)}(y)\,dy \right)  \si \rho_{(i-1)\Delta_n}
\\ &= \label{1.4}\Delta_n g^{(1)}(\zeta) g^{(2)}(\zeta) 1_{\{\zeta \in [0,\Delta_n]\}} \si \rho_{(i-1)\Delta_n} \overset{\text{a.s.}}{\rightarrow} g^{(1)}(0^+) g^{(2)}(0^+) \int_0^1 \rho_s\,ds. 
\end{align}
Line \eqref{1.2} follows by setting $y=i\Delta_n-s$, while line \eqref{1.4} follows from the mean value theorem and the fact that the paths of $\rho$ are almost surely Riemann integrable.
\end{proof}

This concludes the proof of Theorem \ref{gizerotheo}, that deals with the case where the  indices of regular variation satisfy: $\delta\tot1+\delta\tot2\geq 0$.

In the following section, we wish to lift such a restriction, and prove another law of large numbers, this time for a scaled version of the realised covariation. 

\subsection{Proof of Theorem \ref{weaktheonovol}}
As in the last section, thanks to Remark \ref{ucpremark},  we will prove convergence in probability for $t=1$.
In order to prove this result, we denote by $\mathscr F^{\rho}:=\sigma\{\rho_s,s\in\R\}$ and we start by considering the following decomposition:
\begin{align}\begin{split}
&\Delta_n\frac{\sum_{i=1}^n \DD Y\tot1 \DD Y^{(2)}}{c(\Delta_n)}- \int_0^1\rho_l\,dl
=D_n^{(1)}+D_n^{(2)}, \text{ where } 
\\
D_n^{(1)}&=\Delta_n\frac{\sum_{i=1}^n \DD Y\tot1 \DD Y^{(2)}}{c(\Delta_n)}-\si\E\left[ \frac{\Delta_n}{c(\Delta_n)}\DD Y\tot1 \DD Y^{(2)} \Bigg| \mathscr F^{\rho} \right], \\
D_n^{(2)}&= \si\E\left[ \frac{\Delta_n}{c(\Delta_n)}\DD Y\tot1 \DD Y^{(2)} \Bigg| \mathscr F^{\rho} \right]- \int_0^1\rho_l\,dl.
\end{split}
\label{decompose}
\end{align}

\subsubsection{Convergence of $D_n^{(2)}$}
Note that 
\begin{align*}
\E\left[ \DD Y\tot1 \DD Y^{(2)} \Bigg| \mathscr F^{\rho} \right]&=\int_{-\infty}^{i\Delta_n}\phi_{\Delta_n}^{(1)}(i\Delta_n-s)\phi_{\Delta_n}^{(2)}(i\Delta_n-s)\rho_s\,ds\\
&=\int_{0}^{\infty}\phi_{\Delta_n}^{(1)}(s)\phi_{\Delta_n}^{(2)}(s)\rho_{i\Delta_n-s}\,ds.
\end{align*}
Hence we can write
\begin{align}\begin{split}
\label{convergenceeq}
\Delta_n\frac{\sum_{i=1}^n \E\left[\DD Y^{(1)} \DD Y^{(2)} \Bigg| \mathscr F^{\rho}\right]}{c(\Delta_n)}&=\frac{\int_0^\infty\phi_{\Delta_n}^{(1)}(s)\phi_{\Delta_n}^{(2)}(s)\Delta_n\left(\sum_{i=1}^n\rho_{i\Delta_n-s}\right)\,ds}{c(\Delta_n)}\\
&=
\int_{\R^+}  \Delta_n\left(\sum_{i=1}^n\rho_{i\Delta_n-s}\right) \,d\pi_n(s),
\end{split}
\end{align}
where, for each $n\in \mathbb{N}$, $\pi_n$ is a measure whose density with respect to the Lebesgue measure is given by $$\frac {\phi_{\Delta_n}^{(1)}(s)\phi_{\Delta_n}^{(2)}(s)}{\int_0^\infty \phi_{\Delta_n}^{(1)}(s)\phi_{\Delta_n}^{(2)}(s) \,ds}.$$

Note however that $\pi_n(ds)$ is a probability measure if both $g^{(1)}$ and $g^{(2)}$ are positive and for $s>\Delta_n$ one has $\left(g^{(1)}(s)-g^{(1)}(s-\Delta_n)\right)\times\left(g^{(2)}(s)-g^{(2)}(s-\Delta_n)\right)\geq0$, that is, if both are monotonically increasing or decreasing.
This is why we needed to introduce Assumption \ref{monotonic}.

Let us now formulate a key result needed for establishing the  convergence of $D_n^{(2)}$.

\begin{proposition}
\label{newconvergenceprop} {Assume that the assumptions of Theorem \ref{weaktheonovol} hold.}
If the probability measures $\pi_n$ defined after equation \eqref{convergenceeq} converge weakly to a probability $\pi$ on $[0,\infty)$ then:
\begin{equation}
\label{newconvergence}
\Delta_n\frac{\sum_{i=1}^n \E\left[\DD Y^{(1)} \DD Y^{(2)} \Bigg| \mathscr F^{\rho}\right]}{c(\Delta_n)}\overset{\text{a.s.}}{\rightarrow} \int_{\R^+} \left(\int_{-s}^{1-s}\rho_{l}\,dl\right)\,d\pi(s), \quad \text{ as } n \to \infty.
\end{equation}
\end{proposition}
\begin{proof}
Denote by $
f_n(x,\omega):=\Delta_n\si \rho_{i\Delta_n-x}(\omega)
$ and by 
$
f(x,\omega):=\int_{-x}^{1-x}\rho_l(\omega)\,dl.
$
Furthermore, for a measure $\mu$ on $[0,+\infty)$ and a measurable function $g$: $\mu(g):=\int_{\R^+} g(x)\,d\mu(x)$.
For a.s. $\omega$, $\lim_{n\to\infty} f_n(x,\omega)=f(x,\omega)$.

Let $l_i^*$ be any point in $[-x+(i-1)\Delta_n,-x+i \Delta_n]$.
The error we make at the $n$-th step with a Riemann sum is:
\[
\begin{split}
\abs{f_n-f}&=\left|{\si \int_{-x+(i-1)\Delta_n}^{-x+i\Delta_n} \left(\rho_l - \rho_{l_i^*}\right)\,dl }\right| \leq \si \int_{-x+(i-1)\Delta_n}^{-x+i\Delta_n} K \abs{l-l_i^*}^{\alpha}\,dl \\
&\leq \si \int_{-x+(i-1)\Delta_n}^{-x+i\Delta_n} K \Delta_n^{\alpha}\,dl= K n \Delta_n^{\alpha+1},
\end{split}
\]
for a positive constant $K$, hence:
\[
\left| \pi_n(f_n) - \pi_n(f)\right|\leq K n \Delta_n^{\alpha+1} \pi_n(\R^+)=K n \Delta_n^{\alpha+1}\to 0,
\]
since $\pi_n(\R^+):=\int_{\R^+} d\pi_n(x)=1$.
\end{proof}

\begin{proposition}
\label{newconvergenceprop2}  Assume that the assumptions of Theorem \ref{weaktheonovol} hold. Then $D_n^{(2)}\overset{\text{a.s.}}{\rightarrow} 0$ as $n \to \infty$.
\end{proposition}
\begin{proof}
We remark that checking weak convergence of probability measures is equivalent to checking pointwise convergence of the  corresponding cumulative distribution functions. Hence we write
\[
F_n(x)=\pi_n\left([0,x]\right)=\frac{\int_0^x \phi_{\Delta_n}^{(1)}(s)\phi_{\Delta_n}^{(2)}(s)\,ds}{\int_0^\infty \phi_{\Delta_n}^{(1)}(s)\phi_{\Delta_n}^{(2)}(s)\,ds}.
\]
Now fix $\epsilon > \Delta_n$. Then, for all $x>\epsilon$, one has:
\[
F_n(x)-F_n(\epsilon)=\frac{\int_\epsilon^x \phi_{\Delta_n}^{(1)}(s)\phi_{\Delta_n}^{(2)}(s)\,ds}{\int_0^\infty \phi_{\Delta_n}^{(1)}(s)\phi_{\Delta_n}^{(2)}(s)\,ds}\leq \frac{1}{c(\Delta_n)}\sqrt{ \int_\epsilon^x \left(\phi_{\Delta_n}^{(1)}(s)\right)^2\,ds\int_\epsilon^x \left(\phi_{\Delta_n}^{(2)}(s)\right)^2\,ds}.
\]
Denoting by $b^{(1)}$  a real number such that $\left(g^{(1)'}\right)^2(s)$ is decreasing for $s>b^{(1)}$:
\[
\begin{split}
\int_\epsilon^x \left(\phi_{\Delta_n}^{(1)}(s)\right)^2\,ds&=\int_\epsilon^{b\tot1}  \left(\phi_{\Delta_n}^{(1)}(s)\right)^2\,ds+\int_{b\tot1}^x  \left(\phi_{\Delta_n}^{(1)}(s)\right)^2\,ds\\
&\leq \Delta_n^2\left[(b\tot1-\epsilon)\sup_{[\epsilon,b\tot1]}{\left(g^{(1)'}\right)^2}(s)+\norm{g^{(1)'}}_{L^2[\epsilon,\infty)}\right]=C\Delta^2_n\to 0,
\end{split}
\]
so 
\begin{equation}
\label{weakconv1}
\lim_{n\to \infty} F_n(x)-F_n(\epsilon) =0,
\end{equation}
which implies that the limiting cumulative distribution function will be constant on the sets $(\epsilon,\infty)$, for all $\epsilon >0$, and hence its mass will intuitively concentrate at zero.
We now formalise this observation.

We observe that the family of measures $\{\pi_n\}_{n=1}^\infty$ is easily proven to be tight. Indeed, fix  a small $\zeta>0$, then if $b=\max(b^{(1)},b^{(2)})$, by the above computations:
\[
\int_b^\infty \phi_{\Delta_n}^{(1)}(s)\phi_{\Delta_n}^{(2)}(s)\,ds \leq C \Delta^2_n, 
\]
which is smaller than $\zeta$ for $n$ big enough, say for $n> n(\zeta)$. Hence the compact set $[0,b]$ is such that $\pi_n[0,b]>1-\zeta$ for all $n>n(\zeta)$. The finite collection $\{\pi_i\}_{i=1}^{n(\zeta)}$ is obviously tight, and thus we can conclude that the whole family $\{\pi_n\}$ is tight.

By tightness and Prohorov's theorem, the sequence $\pi_n$ is relatively compact, that is, every subsequence contains a further subsubsequence converging weakly to some probability measure $\Q$ (that could depend on the chosen subsequence). By \eqref{weakconv1}, the only possible weak limit is the unit mass at zero, and hence every subsequence has a subsubsequence converging to $\delta_0$. But  this fact is equivalent to the weak convergence $\pi_n \Rightarrow \delta_0$ (see \cite{billingsley2009convergence}, Theorem 2.6).

Now from \eqref{newconvergence} and the above remark we obtain that 
\[
\Delta_n\frac{\sum_{i=1}^n \E\left[\DD Y^{(1)} \DD Y^{(2)} \Bigg| \mathscr F^{\rho}\right]}{c(\Delta_n)}\overset{\text{a.s.}}{\rightarrow} \int_{\R^+} \left(\int_{-s}^{1-s}\rho_{l}\,dl\right)\,d\delta_0(s)=\int_0^1\rho_l\,dl, \quad \text{ as } n \to \infty,
\]
i.e.~$D_n^{(2)}\overset{\text{a.s.}}{\rightarrow} 0$ as $n\to \infty$.
\end{proof}
\subsubsection{Convergence of $D_n^{(1)}$}
It remains to prove that $D_n^{(1)}$ converges to 0. To this end note that asking that its $L^2$-norm goes to zero is equivalent to showing that
$\text{Var}(D_n^{(1)}| \mathscr F^{\rho})\to 0$ as $n\to \infty$.

\begin{proposition}
\label{summingup}
Suppose that the assumptions  of Theorem \ref{weaktheonovol} hold. Then:
\[
\mathrm{Var}\left(\Delta_n\frac{\sum_{i=1}^n \DD Y^{(1)} \DD Y^{(2)}}{c(\Delta_n)} \Bigg| \mathscr F^{\rho} \right) \to 0,\quad \text{ as } n \to \infty.
\]
\end{proposition}

\begin{proof} We can write
\begin{align}
\label{iosdyrgfea}
\text{Var}\left(\Delta_n\frac{\sum_{i=1}^n \DD Y^{(1)} \DD Y^{(2)}}{c(\Delta_n)} \Bigg| \mathscr F^{\rho} \right)=\nonumber\frac{\Delta_n^2}{c^2(\Delta_n)}\left( \si\text{Var}\left(\DD Y^{(1)} \DD Y^{(2)}\Bigg|\mathscr F^\rho\right)\right)\\+2\frac{\Delta_n^2}{c^2(\Delta_n)} \left(\si \sum_{j=i+1}^n \text{Cov}\left(\DD Y^{(1)} \DD Y^{(2)},\Delta^n_j Y\tot1 \Delta^n_j Y\tot2\Bigg|\mathscr F^\rho\right)\right).
\end{align}
Now observe that:
\[
\text{Var}\left( \DD Y^{(1)} \DD Y^{(2)} \Bigg| \mathscr F^{\rho} \right) \leq \E\left[\left( \DD Y\tot1 \DD Y\tot2\right)^2 \Bigg| \mathscr F^{\rho} \right]\sim K  c^2(\Delta_n),
\]
for some positive constant $K>0$ thanks to Assumption \ref{veryuseful} and the computations in the previous sections.
Hence the first term in \eqref{iosdyrgfea} is asymptotically equivalent to:
\[
Kn\Delta_n^2 \to 0.
\]
Now we have that 
\begin{align}\begin{split}
\label{nusvcacrncr}
&\text{Cov}\left(\DD Y^{(1)} \DD Y^{(2)},\Delta^n_j Y\tot1 \Delta^n_j Y\tot2\Bigg|\mathscr F^\rho\right)
=\E\left[\DD Y^{(1)} \DD Y^{(2)}\Delta^n_j Y\tot1 \Delta^n_j Y\tot2\Bigg| \mathscr F^{\rho}\right]\\
& \quad-\E\left[\DD Y^{(1)} \DD Y^{(2)}\Bigg| \mathscr F^{\rho}\right]\E\left[\Delta^n_j Y\tot1 \Delta^n_j Y\tot2\Bigg| \mathscr F^{\rho}\right].
\end{split}
\end{align}
Thanks to Lemma \ref{importantlemma}, we know how to express the first expectation in \eqref{nusvcacrncr} and write
\begin{multline}
\label{sdvnnlas}
\text{Cov}\left(\DD Y^{(1)} \DD Y^{(2)},\Delta^n_j Y\tot1 \Delta^n_j Y\tot2\Bigg|\mathscr F^\rho\right)=\\
\E\left[\DD Y^{(1)}\Delta_j^nY\tot1\big|\mathscr F^\rho\right]\E\left[\DD Y\tot2 \Delta_j^nY\tot2\big|\mathscr F^\rho\right]\\+\E\left[\DD Y\tot1 \Delta_j^nY\tot2\big|\mathscr F^\rho\right]\E\left[\DD Y\tot2 \Delta_j^nY\tot1\big|\mathscr F^\rho\right].
 \end{multline}
Now, recall the notation $\DD g^{(k)}_s:= g^{(k)}\left(i\Delta_n-s\right)-g^{(k)}\left(\qaz-s\right)$, for $k\in \{1, 2\}$. Since $j>i$, we have:
\begin{multline}
\label{ashzzaz}
\E\left[\DD Y^{(1)}\Delta_j^nY\tot1\big|\mathscr F^\rho \right]=\DDint g\tot1 (i\Delta_n-s)\JJ g_s\tot1\,ds+ \int_{-\infty}^\qaz \DD g_s\tot1 \JJ g_s \tot1\,ds
\\
=\int_0^{\Delta_n} g\tot1 (s)\left(g\tot1((j-i)\Delta_n+s)-g\tot1((j-i-1)\Delta_n+s)\right)\,ds+\\
  \int_0^\infty \left(g\tot 1(\Delta_n+s)-g\tot1(s)\right)\left(g\tot1 ((j-i)\Delta_n+s)-g\tot1 ((j-i-1)\Delta_n+s)\right)\,ds,
\end{multline}
and similarly:
\[
\E\left[\DD Y\tot1 \Delta_j^nY\tot2\big|\mathscr F^\rho\right]=\DDint g\tot1(i\Delta_n-s)\DD g_s\tot2 \rho_s\,ds+ \int_{-\infty}^\qaz \DD g_s\tot1 \JJ g_s \tot2\rho_s\,ds,
\]
and analogously for the other terms in \eqref{sdvnnlas}.

These terms do not simply depend on the difference $j-i$ because of the presence of the correlation process $\rho$. Nevertheless, $\rho$ is bounded by 1 in absolute value, so passing to the absolute values, we can  write:
\begin{multline}
 \left|\text{Cov}\left(\DD Y^{(1)} \DD Y^{(2)},\Delta^n_j Y\tot1 \Delta^n_j Y\tot2\Bigg|\mathscr F^\rho\right)\right|\\
 \leq\E\left[\left|\DD Y^{(1)}\Delta_j^nY\tot1\right|\big|\rho\equiv1\right]\E\left[\left|\DD Y\tot2 \Delta_j^nY\tot2\right|\big|\rho\equiv1\right]\\+\E\left[\left|\DD Y\tot1 \Delta_j^nY\tot2\right|\big|\rho\equiv1\right]\E\left[\left|\DD Y\tot2 \Delta_j^nY\tot1\right|\big|\rho\equiv1\right] .
 \end{multline}
The proof that the quantity goes to zero follows from Theorem 4.1 
in \cite{GranelliVeraart2017b}. 
Details on how to apply that theorem in this setting are as follows.

Denote by 
$r\tot n_{a,b}(j-i)=\E\left[\left|\DD Y^{(a)}\Delta_j^nY^{(b)}\right|\big|\rho\equiv1\right]$, for $a, b \in \{1, 2\}$.
Then, in fact,  $r\tot n_{a,b}(j-i)=\E\left[\DD G^{(a)}\Delta_j^nG^{(b)}\right]$, for $a, b \in \{1, 2\}$.
The second term in line \eqref{iosdyrgfea}  is then dominated by:
\begin{align}
&2\frac{\Delta_n^2}{c^2(\Delta_n)}\si \sum_{j=i+1}^n \left[r\tot n_{1,1}(j-i)r\tot n_{2,2}(j-i)+r\tot n_{1,2}(j-i)r\tot n_{2,1}(j-i)\right] \\
&=2\frac{\Delta_n^2}{c^2(\Delta_n)}\si \sum_{j-i=1}^n\left[r\tot n_{1,1}(j-i)r\tot n_{2,2}(j-i)+r\tot n_{1,2}(j-i)r\tot n_{2,1}(j-i)\right]\\
&=2\frac{\Delta_n^2}{c^2(\Delta_n)}n \si \left[r\tot n_{1,1}(i)r\tot n_{2,2}(i)+r\tot n_{1,2}(i)r\tot n_{2,1}(i)\right]\\
&=2\frac{1}{n} \si \left[\frac{r\tot n_{1,1}(i)r\tot n_{2,2}(i)}{c^2(\Delta_n)}+\frac{r\tot n_{1,2}(i)r\tot n_{2,1}(i)}{c^2(\Delta_n)}\right].
\end{align}
With our assumptions, Theorem 4.3 
in \cite{GranelliVeraart2017b}
gives us the uniform bound for a positive constant $C$: 
{
\begin{align*}
\frac{|\mathbb{E}(\Delta_1^nG^{(a)} \Delta_{1+i}^n G^{(b)})|}{\tau_n^{(a)}\tau_n^{(b)}}
\leq C (i-1)^{\delta\tot a+ \delta \tot b +\epsilon -1},
\end{align*}
for all $\epsilon >0$, 
where, for $i\in \{a, b\}$: 
\begin{align*}\tau _n \tot i &:=\sqrt{ \E\left[ \left(\Delta_1^n G\tot i\right)^2\right]}=\sqrt{\int_0^{\infty} \left(g\tot i(s+\Delta_n)-g\tot i(s)\right)^2\,ds + \int_0^{\Delta_n} \left(g\tot i (s)\right)^2\,ds}\\
&=\sqrt{\bar R^{(i)}(\Delta_n)}=\Delta_n^{\delta^{(i)}+1/2}\sqrt{L_0^{(i)}(\Delta_n)}.
\end{align*}
We can write
\begin{align*}
\frac{\left| r^{(n)}_{a,b}(k)\right|}{c(\Delta_n)}=\frac{|\mathbb{E}(\Delta_1^nG^{(a)} \Delta_{1+k}^n G^{(b)})|}{c(\Delta_n)}
=
\frac{|\mathbb{E}(\Delta_1^nG^{(a)} \Delta_{1+k}^n G^{(b)})|}{\tau_n^{(a)}\tau_n^{(b)}} 
\frac{\tau_n^{(a)}\tau_n^{(b)}}{c(\Delta_n)}.
\end{align*}
According to Assumption  \ref{terrific}, we have
\begin{align*}
\frac{\tau_n^{(a)}\tau_n^{(b)}}{c(\Delta_n)}
= \frac{\Delta_n^{\delta^{(a)}+\delta^{(b)}+1}\sqrt{L_0^{(a)}(\Delta_n) L_0^{(b)}(\Delta_n)}}{\Delta_n^{\delta^{(1)}+\delta^{(2)}+1}L_4^{(1,2)}(\Delta_n)}
=\frac{\sqrt{L_0^{(a)}(\Delta_n) L_0^{(b)}(\Delta_n)}}{L_4^{(1,2)}(\Delta_n)}.
\end{align*}
Hence
\begin{align*}
\frac{\tau_n^{(1)}\tau_n^{(1)}}{c(\Delta_n)}
\frac{\tau_n^{(2)}\tau_n^{(2)}}{c(\Delta_n)}
= \frac{L_0^{(1)}(\Delta_n) L_0^{(2)}(\Delta_n)}{(L_4^{(1,2)}(\Delta_n))^2}=\frac{\tau_n^{(1)}\tau_n^{(2)}}{c(\Delta_n)}
\frac{\tau_n^{(2)}\tau_n^{(1)}}{c(\Delta_n)}.
\end{align*}
According to Assumption \ref{terrific} the term $\frac{L_0^{(1)}(\Delta_n) L_0^{(2)}(\Delta_n)}{(L_4^{(1,2)}(\Delta_n))^2}$ 
converges, so it can be bounded by a constant $K$.
Hence we can deduce that for all $\epsilon > 0$ there exists an $n_0(\epsilon)\in \N$ such that  
\begin{align*}
\frac{r\tot n_{1,1}(i)r\tot n_{2,2}(i)}{c(\Delta_n)^2}\leq KC(i-1)^{2\delta\tot1+2\delta\tot2+2\epsilon-2}, && \frac{r\tot n_{1,2}(i)r\tot n_{2,1}(i)}{c(\Delta_n)^2} \leq KC(i-1)^{2\delta\tot1+2\delta\tot2+2\epsilon-2}
\end{align*}
 and for all $n\geq n_0(\epsilon)$.
}
Convergence to zero then follows from  Ces\'{a}ro's Theorem and from:
\[
\lim_{i\to \infty} (i-1)^{2\delta\tot1+2\delta\tot2+2\epsilon-2} = 0  {\iff}  2\delta\tot1+2\delta\tot2+2\epsilon-2<0 \iff   \epsilon< 1-\delta\tot1-\delta\tot2,
\]
which is always possible, as $\delta\tot 1+\delta\tot 2\in(-1,+1)$.

\end{proof}

\subsection{Proof of Theorem \ref{conv2}}
We first observe that, thanks to Theorem \ref{ucpsuff} and Remark \ref{ucpremark}, $u.c.p.$ convergence on $[0,T]$ is implied by convergence in probability for all $t\in[0,T]$. Without loss of generality, we will then restrict ourselves to $t=1$.
The proof will be split in several steps.
We start with the familiar splitting:
\begin{align}\begin{split}
&\Delta_n\frac{\sum_{i=1}^n \DD X\tot1 \DD X^{(2)}}{c(\Delta_n)}- \int_0^1\sigma\tot1_l\sigma\tot2_l\rho_l\,dl
=G_n^{(1)}+G_n^{(2)}, \text{ where } 
\\
G_n^{(1)}&=\Delta_n\frac{\sum_{i=1}^n \DD X\tot1 \DD X^{(2)}}{c(\Delta_n)}-\si\E\left[ \frac{\Delta_n}{c(\Delta_n)}\DD X\tot1 \DD X^{(2)} \Bigg| \mathscr H \right], \\
G_n^{(2)}&= \si\E\left[ \frac{\Delta_n}{c(\Delta_n)}\DD X\tot1 \DD X^{(2)} \Bigg| \mathscr H \right]- \int_0^1\sigma\tot1_l\sigma\tot2_l\rho_l\,dl.
\end{split}
\label{decompose2}
\end{align}
Proposition \eqref{newconvergenceprop} easily extends to this situation and we can show that $G_n^{(2)}\stackrel{\text{a.s.}}{\to} 0$, as $n \to \infty$:

\begin{proposition}
Let $\sigma\tot1,\sigma\tot2,\rho$ have H\"older continuous sample paths, almost surely.
Let $\pi_n$ be a sequence of probability measures converging weakly to the probability $\pi$.
Then, the almost sure convergence of the Riemann sums
\[
f_n(\omega,x):=\si\rho_{i\Delta_n-x}\sigma\tot1_{i\Delta_n-x}\sigma\tot2_{i\Delta_n-x} \to \int_{-x}^{1-x} \sigma\tot1_l\sigma\tot2_l\rho_l\,dl=:f(\omega,x)
\]
implies the convergence of the integrals:
\[
\pi_n(f_n)\to\pi(f).
\]
In particular, we have the almost sure convergence:
\[
\si\E\left[ \frac{\Delta_n}{c(\Delta_n)}\DD X\tot1 \DD X^{(2)} \Bigg| \mathscr H \right] \stackrel{a.s.}{\to}\int_{\R^+} \left(\int_{-x}^{1-x}\sigma\tot1_l\sigma\tot2_l\rho_l\,dl\right) \, d\pi(x),\quad \text{ as } n \to \infty.
\]
\end{proposition}
 
\begin{proof}
Let  $\alpha_\rho, \alpha_{\sigma\tot1},\alpha_{\sigma\tot2}$ denote the H\"older coefficients of $\rho, \sigma^{(1)}, \sigma^{(2)}$ and set  $\alpha:=\min\{\alpha_\rho,\alpha_{\sigma\tot1},\alpha_{\sigma\tot2}\}$, and $|u-t|<1$. Then 
\[
\begin{split}
\left|\sigma_u\tot1\sigma\tot2_u\rho_u-\sigma\tot1_t\sigma_t\tot2\rho_t\right|&\leq\left|\sigma\tot1_u \sigma\tot2_u (\rho_u-\rho_t) + \rho_t (\sigma\tot1_u\sigma\tot2_u-\sigma\tot1_t\sigma\tot2_t)\right|\\
&\leq \left|\sigma\tot1_u \sigma\tot2_u (\rho_u-\rho_t)+\rho_t \left(\sigma\tot1_u\left(\sigma\tot2_u-\sigma\tot2_t+\sigma\tot2_t\right)-\sigma\tot1_t\sigma\tot2_t\right)\right|\\
&\leq\left|u-t\right|^{\alpha_\rho}|\sigma\tot1_u\sigma\tot2_u|+|\sigma\tot1_u||u-t|^{\alpha_{\sigma\tot2}}+|\sigma\tot2_t||u-t|^{\alpha_{\sigma\tot1}} \\
&\leq |u-t|^\alpha \left( |\sigma\tot1_u\sigma\tot2_u|+|\sigma\tot1_u|+|\sigma\tot2_t|\right).
\end{split}
\]
The error we make with a Riemann sum becomes:
\[
\begin{split}
&\abs{f_n(x)-f(x)}=\left|\Delta_n\si\rho_{i\Delta_n-x}\sigma\tot1_{i\Delta_n-x}\sigma\tot2_{i\Delta_n-x}-\int_{-x}^{1-x}\rho_l\sigma\tot1_l\sigma\tot2_l\,dl\right|\\
&\leq\si \int_{-x+(i-1)\Delta_n}^{-x+i\Delta_n} \abs{\sigma\tot1_l\sigma\tot2_l\rho_l - \sigma\tot1_{l_i^*}\sigma\tot2_{l_i^*}\rho_{l_i^*}}\,dl \\
&\leq \si \int_{-x+(i-1)\Delta_n}^{-x+i\Delta_n}  \abs{l-l_i^*}^{\alpha}\left( |\sigma\tot1_l\sigma\tot2_l|+|\sigma\tot1_l|+|\sigma\tot2_{l_i^*}|\right)\,dl \\
&\leq\Delta_n^\alpha \left(\int_{-x}^{1-x} \left(|\sigma\tot1_l\sigma\tot2_l|+|\sigma\tot1_l|\right)\,dl + \Delta_n \si |\sigma\tot2_{l_i^*}| \right).
\end{split}
\]
Now we can finally prove the bound for the integrals:
\[
\begin{split}
|\pi_n(f_n)-\pi_n(f)|&\leq \Delta_n^\alpha \pi_n\left(\int_{-x}^{1-x} \left(|\sigma\tot1_l\sigma\tot2_l|+|\sigma\tot1_l|\right)\,dl+
\Delta_n \si |\sigma\tot2_{l_i^*}|\right)\\
&=\Delta_n^\alpha\pi_n\left(\int_{-x}^{1-x} \left(|\sigma\tot1_l\sigma\tot2_l|+|\sigma\tot1_l|\right)\,dl\right)+\Delta_n^\alpha\pi_n\left(\Delta_n \si |\sigma\tot2_{l_i^*}|\right).
\end{split}
\]
The first term converges to $\pi\left(\int_{-x}^{1-x}  |\sigma\tot1_l\sigma\tot2_l|+|\sigma\tot1_l|\,dl\right)$, since the function $x\to \int_{-x}^{1-x}  |\sigma\tot1_l\sigma\tot2_l|+|\sigma\tot1_l|\,dl$ is continuous and tends to zero for $x\to \infty$.
Since the paths of $\sigma\tot2$ are almost surely H\"older continuous, Proposition \ref{newconvergenceprop} applies, and we can conclude that:
\[
\pi_n\left(\Delta_n \si |\sigma\tot2_{l_i^*}|\right)\to \pi\left(\int_{-x}^{1-x}|\sigma\tot2_l|\,dl\right);
\]
henceforth:
\[
|\pi_n(f_n)-\pi_n(f)|\to 0 \times \pi\left(\int_{-x}^{1-x} |\sigma\tot1_l\sigma\tot2_l|+|\sigma\tot1_l|+|\sigma\tot2_l|\,dl\right)=0.
\]
\end{proof}
Analogously to Proposition \ref{summingup}, we  need to show that:
\[
\Var\left(\frac{\Delta_n}{c(\Delta_n)}\si\DD X\tot1 \DD X^{(2)} \Bigg| \mathscr H \right)\to 0, \quad \text{ as } n \to \infty,
\]
which implies that $G_n^{(1)}\stackrel{L^2}{\to} 0$ as $n\to \infty$.
\begin{proposition}
Suppose that the assumption of Theorem  \ref{FullStory} hold. Then 
\[
\Var\left(\frac{\Delta_n}{c(\Delta_n)}\si\DD X\tot1 \DD X^{(2)} \Bigg| \mathscr H \right)\to 0, \quad \text{ as } n \to \infty,
\]
\end{proposition}
\begin{proof}[Sketch of Proof]
We will only sketch the proof, which follows the same lines as the proof of Proposition \ref{summingup}. The basic idea is the following: Since the volatility process is c\`adl\`ag, it is bounded on compact intervals, while unbounded intervals are controlled via Assumption \ref{strangevolassumption}. 

All the terms concerning intervals on the sets of the form $[(i-1)\Delta_n,i\Delta_n]$ of the proof of Proposition \ref{summingup} are treated exactly in the same way, since we can  uniformly bound $\sigma$ by its maximum on $[0,1]$. When $\sigma$ appears integrated on an interval of infinite length, bounded away from 0, we can write for $j\in \{1, 2\}$:
\begin{align}
&\int_{b\tot j}^{\infty}\left(g\tot j(s+\Delta_n)-g\tot j (s)\right)^2\left(\sigma\tot j_{(i-1)\Delta_n-s}\right)^2\,ds \leq\\
& \Delta_n^2\int_{b\tot j}^\infty \left(\frac{d}{ds}g\tot j(s)\right)^2\left(\sigma\tot j_{(i-1)\Delta_n-s}\label{lnsctn}\right)^2\,ds=O(\Delta_n^2),
\end{align}
since the integral in \eqref{lnsctn} is finite, thanks to Assumption \ref{strangevolassumption}.

\end{proof}

\begin{appendix}
\section{Background results}\label{A}
\subsection{\emph{U.c.p.} convergence of processes}\label{Sectucp}
We briefly recall the mode of convergence that we work with throughout this article: \emph{uniform convergence on compacts  in probability}, or \emph{u.c.p.} for short.

\begin{definition}[$u.c.p.$ convergence]
\label{ucpdef}
The sequence of \cadlag processes $X\tot n$ is said to converge \emph{uniformly on compacts in probability} to the \cadlag process $X$  if, for all $t >0$ and all $\epsilon >0$:
\[
\lim_{n\to \infty} \Prob\left(  \sup_{s<t}\left| X_s\tot n- X_s\right| > \epsilon\right)= 0
\]
\end{definition}
\begin{remark}
Note that the assumption that the processes are \cadlag is sufficient to ensure that the supremum is $\Prob$-measurable.
\end{remark}

In some cases, pointwise convergence in probability is enough to prove u.c.p.~convergence of processes as can be seen in the following theorem.
\begin{theorem}[Sufficient condition for u.c.p.~convergence]
\label{ucpsuff}
Let the assumption and notation of Definition \ref{ucpdef} prevail.
Suppose that, for all $t$ in a dense subset $D\subset \R^+$ we have $X \tot n _t\overset{\Prob}{\rightarrow}X_t$. 
Assume further, that the paths of $X\tot n$ are increasing with time and the paths of $X$ are continuous, almost surely.
Then, the (stronger) convergence
$X\tot n_\cdot \overset{\text{u.c.p.}}{\rightarrow} X_\cdot$
holds.
\end{theorem}

\subsection{Expectation of the product of four stochastic integrals}
\begin{lemma}\label{importantlemma}
Let $W$ denote a standard Brownian motion and let $H^{(i)}(t)$, for $i \in \{1,2,3,4\}$ be  deterministic functions such that $\int_0^T \left(H^{(i)}(t)\right)^2\,dt <\infty$.    Then, one has:
\begin{multline}
\label{important}
\Et\int_0^T H^{(1)}_s\,dW_s\int_0^T H^{(2)}_s\,dW_s \int_0^T H^{(3)}_s\,dW_s \int_0^T H^{(4)}_s\,dW_s\right]\\
=\int_0^T H^{(1)}_sH^{(3)}_s\,ds\int_0^T H^{(2)}_sH^{(4)}_s\,ds+ \int_0^T H^{(1)}_sH^{(2)}_s\,ds\int_0^T H^{(3)}_sH^{(4)}_s\,ds
\\
+\int_0^T H^{(1)}_sH^{(4)}_s\,ds\int_0^T H^{(2)}_sH^{(3)}_s\,ds.
\end{multline}
\label{finallyaproof}
\end{lemma}
\begin{proof}
We will prove the claim by differentiating the characteristic function. In details, we compute:
\begin{equation}
\label{characteristic1}
\frac{\partial^4}{\partial\theta_4\partial\theta_3\partial\theta_2\partial\theta_1}\Bigg|_{\boldsymbol\theta=0}\E\left[e^{i\theta_1 \int_0^T H^{(1)}_s\,dW_s+i\theta_2 \int_0^T H^{(2)}_s\,dW_s +\theta_3 \int_0^T H^{(3)}_s\,dW_s+i\theta_4 \int_0^T H^{(4)}_s\,dW}\right],
\end{equation}
for $\boldsymbol \theta=(\theta_1,\theta_2,\theta_3,\theta_4)^\top.$
The random variable
\[
\theta_1 \int_0^T H\tot 1_s\,dW_s+\theta_2 \int_0^T H\tot 2_s\,dW_s+\theta_3 \int_0^T H\tot 3_s\,dW_s+\theta_4 \int_0^T H\tot 4_s\,dW_s
\]
follows a Gaussian distribution with mean zero and variance:
\[
\int_0^T \left(\theta_1 H\tot1_s+\theta_2 H\tot2_s+\theta_3 H\tot3_s+\theta_4 H\tot4_s\right)^2\,ds.
\]
Note that the variable has moments of all orders, hence our approach is justified.

The expectation in \eqref{characteristic1} equals:
\[\tilde H (\theta_1,\theta_2,\theta_3,\theta_4) :=
\exp\left(-\frac 12 \int_0^T \left(\theta_1 H\tot1_s+\theta_2 H\tot2_s+\theta_3 H\tot3_s+\theta_4 H\tot4_s\right)^2\right)\,ds.
\]
We proceed with differentiating the function $\tilde H$.
\[
\frac{\partial \tilde H}{\partial \theta_1}=- \tilde H  \int_0^T \left(\theta_1 H\tot1_s+\theta_2 H\tot2_s+\theta_3 H\tot3_s+\theta_4 H\tot4_s\right) H\tot1_s\,ds
\]
\[
\begin{split}
\frac{\partial^2 \tilde H}{\partial \theta_2\partial\theta_1}=-\tilde H \int_0^T H\tot1_sH\tot2_s\,ds+\tilde H \Bigg[&\int_0^T \left(\theta_1 H\tot1_s+\theta_2 H\tot2_s+\theta_3 H\tot3_s+\theta_4 H\tot4_s\right) H\tot1_s\,ds\times\\
&  \int_0^T \left(\theta_1 H\tot1_s+\theta_2 H\tot2_s+\theta_3 H\tot3_s+\theta_4 H\tot4_s\right) H\tot2_s\,ds \Bigg]
\end{split}
\]
\[
\begin{split}
\frac{\partial^3\tilde H}{\theta_3\theta_2\theta_1}&=\tilde H  \int_0^T H_s\tot1 H\tot 2_s\,ds \int_0^T \left(\theta_1 H\tot1_s+\theta_2 H\tot2_s+\theta_3 H\tot3_s+\theta_4 H\tot4_s\right) H\tot3_s\,ds\\
&+\tilde H  \Bigg[\int_0^T \left(\theta_1 H\tot1_s+\theta_2 H\tot2_s+\theta_3 H\tot3_s+\theta_4 H\tot4_s\right)H\tot1_s\,ds\\
&\times\int_0^T \left(\theta_1 H\tot1_s+\theta_2 H\tot2_s+\theta_3 H\tot3_s+\theta_4 H\tot4_s\right)H\tot2_s\,ds \\&\times \int_0^T \left(\theta_1 H\tot1_s+\theta_2 H\tot2_s+\theta_3 H\tot3_s+\theta_4 H\tot4_s\right)H\tot3_s\,ds\Bigg]\\
&+\tilde H  \Bigg[\int_0^TH\tot1_sH\tot3_s\,ds \int_0^T \left(\theta_1 H\tot1_s+\theta_2 H\tot2_s+\theta_3 H\tot3_s+\theta_4 H\tot4_s\right) H\tot2_s\,ds\\
&+ \int_0^TH\tot2_sH\tot3_s\,ds \int_0^T \left(\theta_1 H\tot1_s+\theta_2 H\tot2_s+\theta_3 H\tot3_s+\theta_4 H\tot4_s\right) H\tot1_s\,ds\Bigg].
\end{split}
\]
And finally, since $\tilde H(0,0,0,0)=1$ and $\frac{\partial \tilde H}{\partial \theta_i}\big|_{\boldsymbol \theta=0}=0$,
\begin{multline}
\frac{\partial^4 \tilde H}{\partial\theta_4\partial\theta_3\partial\theta_2\partial\theta_1}\Bigg|_{\boldsymbol\theta=0} =\int_0^T H\tot1_sH\tot2_s\,ds\int_0^T H\tot3_sH\tot4_s\,ds+\\
\int_0^T H\tot1_sH\tot3_s\,ds\int_0^T H\tot2_sH\tot4_s\,ds+\int_0^T H\tot2_sH\tot3_s\,ds\int_0^T H\tot1_sH\tot4_s\,ds,
\end{multline}
proving the claim.
\end{proof}
A particularly useful consequence of Lemma \ref{finallyaproof} is the following result:
\begin{corollary}
\label{corollary5}
Let $\left(K_t\right)_{t\geq 0}$ be an a.s. bounded stochastic process, then:
\begin{multline*}
\E\left[ \left( \DDint g^{(1)}(i\Delta_n-s)\,dW^{(1)}_s \DDint g^{(2)}(i\Delta_n-s)K_s\,dW^{(1)}_s\right)^2\Bigg|\mathscr F^K\right]\\
=  \DDint \left(g\tot1 (i\Delta_n-s)\right)^2\,ds\DDint \left(g\tot2(i\Delta_n-s)\right)^2K_s^2\,ds \\
+2\left(\DDint g\tot1 (i\Delta_n-s)g\tot2 (i\Delta_n-s) K_s\,ds\right)^2.
\end{multline*}
\end{corollary}
\end{appendix}

\section*{Acknowledgement}
We wish to thank Damiano Brigo, Dan Crisan, Mikko Pakkanen and  Mark Podolskij   for helpful discussions. 
AG is grateful to the Department of Mathematics of Imperial College for his PhD scholarship which supported this research. AEDV acknowledges
financial support by a Marie Curie FP7 Integration Grant within the 7th European Union Framework Programme (grant agreement number PCIG11-GA-2012-321707).


\end{document}